\documentclass[oneside,english]{amsart}
\usepackage[T1]{fontenc}
\usepackage[latin9]{inputenc}
\usepackage{textcomp}
\usepackage{amstext}
\usepackage{amsthm}
\usepackage{amssymb}
\usepackage{esint}

\makeatletter
\numberwithin{equation}{section}
\numberwithin{figure}{section}
\theoremstyle{plain}
\newtheorem{thm}{\protect\theoremname}[section]
  \theoremstyle{plain}
  \newtheorem{cor}[thm]{\protect\corollaryname}
  \theoremstyle{definition}
  \newtheorem{example}[thm]{\protect\examplename}
  \theoremstyle{plain}
  \newtheorem{prop}[thm]{\protect\propositionname}
  \theoremstyle{remark}
  \newtheorem{rem}[thm]{\protect\remarkname}
  \theoremstyle{plain}
  \newtheorem{lem}[thm]{\protect\lemmaname}
  \theoremstyle{definition}
  \newtheorem{defn}[thm]{\protect\definitionname}
  \theoremstyle{plain}
  \newtheorem{conjecture}[thm]{\protect\conjecturename}

\def\makebbb#1{
    \expandafter\gdef\csname#1\endcsname{
        \ensuremath{\Bbb{#1}}}
}\makebbb{R}\makebbb{N}\makebbb{Z}\makebbb{C}\makebbb{H}\makebbb{E}\makebbb{H}\makebbb{P}\makebbb{B}\makebbb{Q}\makebbb{E}

\usepackage{babel}

\usepackage{babel}

\makeatother

\usepackage{babel}

\makeatother

\usepackage{babel}
  \providecommand{\conjecturename}{Conjecture}
  \providecommand{\corollaryname}{Corollary}
  \providecommand{\definitionname}{Definition}
  \providecommand{\examplename}{Example}
  \providecommand{\lemmaname}{Lemma}
  \providecommand{\propositionname}{Proposition}
  \providecommand{\remarkname}{Remark}
\providecommand{\theoremname}{Theorem}

\begin{document}

\title{Kähler-Einstein metrics, canonical random point processes and birational
geometry}
\begin{abstract}
In the present paper and the companion paper \cite{berm8} a probabilistic
(statistical mechanical) approach to the study of canonical metrics
and measures on a complex algebraic variety $X$ is introduced. On
any such variety with positive Kodaira dimension a canonical (birationally
invariant) random point processes is defined and shown to converge
in probability towards a canonical deterministic measure $\mu_{X}$
on $X,$ coinciding with the canonical measure of Song-Tian and Tsuji.
The proof is based on new large deviation principle for Gibbs measures
with singular Hamiltonians which relies on an asymptotic submean inequality
in large dimensions, proved in a companion paper. In the case of a
variety $X$ of general type we obtain as a corollary that the (possibly
singular) Kähler-Einstein metric on $X$ with negative Ricci curvature
is the limit of a canonical sequence of quasi-explicit Bergman type
metrics. In the opposite setting of a Fano variety $X$ we relate
the canonical point processes to a new notion of stability, that we
call Gibbs stability, which admits a natural algebro-geometric formulation
and which we conjecture is equivalent to the existence of a Kähler-Einstein
metric on $X$ and hence to K-stability as in the Yau-Tian-Donaldson
conjecture. 
\end{abstract}

\author{Robert J. Berman}

\maketitle
\tableofcontents{}

\section{Introduction}

Kähler-Einstein metrics, i.e. Kähler metrics with constant Ricci curvature,
play a prominent role in the study of complex algebraic varieties.
When such a variety $X$ admits a Kähler-Einstein metric it is essentially
unique, i.e. canonically attached to $X$ and can thus be used to
probe the space $X$ using differential-geometric tools (e.g. as in
Yau's proof of the Miyaoka-Yau Chern number inequalities). Singular
versions of Kähler-Einstein are also naturally linked to Mori's Minimal
Model Program (MMP) for complex algebraic varieties, notably through
the Kähler-Ricci flow \cite{s-t,s-t0,ts}. In the case when the variety
$X$ is defined over the integers, i.e. $X$ is cut out by polynomials
with integer coefficients, it is also expected that Kähler-Einstein
metrics carry arithmetic information (even if there are very few direct
results in this direction). For example, the role of Kähler-Einstein
metrics in arithmetic (Arakelov) geometry was speculated on by Manin
\cite{man}, as playing the role of minimal models over the prime/place
at infinity. 

In the present paper and the companion paper \cite{berm8} a new\emph{
}probabilistic\emph{ }(statistical-mechanical) approach to the study
of Kähler-Einstein type metrics is introduced, where the metrics appear
in the large $N-$limit of of certain\emph{ }canonical random point
processes on $X$ with $N$ particles. The canonical point processes
in question are directly defined in terms of algebro-geometric data
and the thrust of this approach is thus that it gives a new direct
link between \emph{algebraic geometry }on one hand and\emph{ complex
differential (Kähler) geometry} on the other. In the companion paper
\cite{berm8} a general convergence result for stochastic interacting
particle systems in thermal equilibrium is established and applied
in the setting when the canonical line bundle $K_{X}$ is positive/ample
to produce the unique Kähler-Einstein metric on $X.$ In the present
paper the results are extended to any variety $X$ of positive Kodaira
dimension, using the global pluripotential theory and variational
calculus in \cite{begz,bbgz,berm6} and some algebraic geometry (in
particular the Fujino-Mori canonical bundle formula \cite{f-m}).
While our main results are centered around the case of negatively
curved Kähler-Einstein metrics and various singular and twisted generalizations
of such metrics, we also formulate a conjectural picture relating
the existence problem for Kähler-Einstein metrics with positive Ricci
curvature on a Fano manifold $X$ to a statistical mechanical notion
of stability that we call \emph{Gibbs stability.} The latter notion
of stability thus replaces the notion of K-stability which appears
in the seminal Yau-Tian-Donaldson conjecture for Kähler-Einstein metrics
on Fano manifolds. Interestingly, the notion of Gibbs stability also
admits a natural purely algebro-geometric formulation in terms of
the standard singularity notions in the MMP. 

The connections to physics (in particular emergent gravity and fermion-boson
correspondences) were emphasized in \cite{berm5}, where a heuristic
argument for the convergence of the point processes was outlined.
In the present paper and its companion \cite{berm8} we provide rigorous
proofs - in a more general setting - building on \cite{b-b,b-b-w,bbgz,berm6}.
The key new technical feature is a mean value inequality for quasi-subharmonic
functions on Riemannian quotients in large dimensions with a distortion
coefficient which is sub-exponential in the dimension proved in \cite{berm8}.

Before turning to precise statements of the main results it may also
be worth pointing out that there are also numerical motivations for
the current approach. Indeed, there are very few examples of Kähler-Einstein
metrics that can be written down explicitly and one virtue of the
present framework is that it offers an, essentially, explicit way
of numerically sampling random points in order to approximate the
Kähler-Einstein metric. It would thus be interesting to compare it
with other recently proposed numerical schemes \cite{do3,d-h-h-k}.

\subsection{Kähler-Einstein geometry and the canonical point processes}

Let $X$ be an $n-$dimensional compact complex manifold. A Kähler
metric $\omega$ on $X$ is said to be \emph{Kähler-Einstein} if its
Ricci curvature is constant: 
\begin{equation}
\mbox{Ric}\ensuremath{\omega_{KE}=-\beta\omega_{KE}}\label{eq:k-e eq}
\end{equation}
where, after normalization one may assume that $\beta=\pm1$ or $\beta=0.$
Since the Ricci form of a Kähler metric represents minus the first
Chern class $c_{1}(K_{X})$ of the canonical line bundle $K_{X}:=\Lambda^{n}(TX)$
of $X$ the Kähler-Einstein equation imposes cohomological conditions
saying that $c_{1}(K_{X})$ vanishes when $\beta=0$ and has a definite
sign when $\beta=\pm1.$ In the latter case, which is the one we shall
mainly focus on, this means in algebro-geometric terms that $\pm K_{X}$
is ample (using additive notation for tensor products, so that $-K_{X}$
denotes the dual of $K_{X})$ and in particular $X$ is a projective
algebraic variety. As is well-known there are also singular versions
of Kähler-Einstein metrics obtained by either relaxing the positivity
(or negativity) condition on $K_{X}$ or introducing a log structure
on $X,$ i.e. a suitable divisor $D$ on $X$ (see below). We recall
that these notions appear naturally in the Minimal Model Program,
which aims at attaching a minimal model to a given algebraic variety
(say with positive Kodaira dimension), i.e. a birational model whose
canonical line bundle is nef (which is the numerical version of semi-positivity)
\cite{bi}. Recently, there has also been a rapid development of the
theory of Kähler-Einstein metrics attached to a log pair $(X,D)$
(see for example \cite{do-3,jmr,cgh,g-p}).

\subsubsection{Varieties of general type ($\beta=1)$}

Let us start with the case when $\beta=1,$ i.e. the case when $K_{X}$
is positive or more generally big (i.e. $X$ is a variety with general
type). We will show how to recover the unique (singular) Kähler-Einstein
metric on such a manifold $X$ from the large $N$ limit of certain\emph{
canonical} random point processes on $X$ with $N$ particles, defined
as follows. First define the following sequence of positive integers:
\[
N_{k}:=\dim H^{0}(X,kK_{X}),
\]
 where $H^{0}(X,kK_{X})$ is the space of all pluricanonical (holomorphic)
$n-$forms of $X,$ at level $k$ (recall that we are using additive
notation for tensor products). In other words, $N_{k}$ is the $k$
th plurigenus of $X$ and we recall that $X$ is said to be of\emph{
general type} if $N_{k}$ is of the order $k^{n}$ for $k$ large.
In particular, this is the case if $K_{X}$ is ample. The starting
point of the present probabilistic approach is the observation that
there is a canonical probability measure $\mu^{(N_{k})}$ on the $N_{k}-$fold
product $X^{N_{k}}$ which may be defined as follows, in terms of
local holomorphic coordinates $z$ on $X:$ 
\begin{equation}
\mu^{(N_{k})}=:=\frac{1}{Z_{N_{k}}}\left|(\det S^{(k)})(z_{1},...,z_{N_{k}})\right|^{2/k}dz_{1}\wedge d\bar{z}_{1}\wedge\cdots\wedge dz_{N_{k}}\wedge d\bar{z}_{N_{k}}\label{eq:canon prob measure intro}
\end{equation}
 where $dz\wedge d\bar{z}$ is a short hand for the local Euclidean
volume form on $X$ determined by the local coordinates $z$ and $\det S^{(k)}$
is a holomorphic section of the line bundle $(kK_{X})^{\boxtimes N_{k}}\rightarrow X^{N_{k}},$
defined by a generator of the determinantal line $\Lambda^{N_{k}}(H^{0}(X,kK_{X})$
(and thus defined up to a multiplicative complex number). Concretely,
we may take $\det S^{(k)}$ as the following Vandermonde type determinant
\[
(\det S^{(k)})(z_{1},z_{2},...,z_{N}):=\det(s_{i}^{(k)}(z_{j})),
\]
 in terms of a given basis $s_{i}^{(k)}$ in $H^{0}(X,kK_{X}),$ which
locally may be identified with a holomorphic function on $X^{N_{k}}.$
The point is that, by the very definition of the canonical line bundle
$K_{X},$ the local function $\left|(\det S^{(k)})(z_{1},...,z_{N_{k}})\right|^{2/k}$
transforms as a density on $X^{N_{k}}$ and after normalization by
its total integral $Z_{N_{k}}$ one obtains a sequence of globally
well-defined probability measures $\mu^{(N_{k})}$ on $X^{N_{k}}$
which are independent of the choice of basis $s_{i}^{(k)}$ in $H^{0}(X,kK_{X})$
and thus canonically attached to $X.$ From a statistical mechanical
point of view the normalizing constant 
\begin{equation}
Z_{N_{k}}:=\int_{X^{N_{k}}}\left|(\det S^{(k)})(z_{1},...,z_{N_{k}})\right|^{2/k}dz_{1}\wedge d\bar{z}_{1}\wedge\cdots\wedge dz_{N_{k}}\wedge d\bar{z}_{N_{k}}\label{eq:def of partition function for kx intro}
\end{equation}
is the\emph{ partition function }and it depends on the choice of generator
(but the point is that $\mu^{(N_{k})}$ does not).

By construction the probability measure $\mu^{(N_{k})}$ is symmetric,
i.e. invariant under the natural action of the permutation group $\Sigma_{N_{k}}$
and hence defines a \emph{random point process on $X$ with $N_{k}$
particles} (i.e. a probability measure on the space of all configurations
of $N_{k}$ points on $X).$ To simplify the notation we will often
omit the subscript $k$ on $N_{k}.$ This should cause no confusion,
since $k$ tends to infinity precisely when $N_{k}$ does. We recall
that the \emph{empirical measure} of a random point process with $N$
particles on a space $X$ is the random measure 
\begin{equation}
\delta_{N}:=\frac{1}{N}\sum_{i=1}^{N}\delta_{x_{i}}\label{eq:empirical measure intro}
\end{equation}
on the probability space $(X^{N},\mu^{(N)})$ which defines a map
from $X^{N}$ to the space $\mathcal{M}_{1}(X)$ of all normalized
measures $\mu$ on $X.$ By definition the\emph{ law} of $\delta_{N}$
is the push-forward of $\mu^{(N)}$ to $\mathcal{M}_{1}(X)$ under
the map $\delta_{N},$ which thus defines a probability measure on
the space $\mathcal{M}_{1}(X).$ 
\begin{thm}
\label{thm:thm gen type intro}Let $X$ be a variety of general type.
Then the empirical measures of the canonical random point processes
on $X$ converge in probability towards the normalized volume form
$dV_{KE}$ of the Kähler-Einstein metric on $X.$ More precisely,
the laws of the empirical measures satisfy a large deviation principle
with speed $N$ and rate functional $F(\mu),$ where $F(\omega^{n}/V)$
is Mabuchi's K-energy of the Kähler form $\omega$ (normalized so
that $F$ vanishes on $dV_{KE})$ 
\end{thm}
We recall that, loosely speaking, the large deviation principle referred
to in the previous theorem says that the convergence in probability
is exponential in the following sense: 
\[
\mbox{Prob }\left(\frac{1}{N_{k}}\sum_{i}\delta_{x_{i}}\in B_{\epsilon}(\mu)\right)\sim e^{-NF(\mu)}
\]
 as first $N\rightarrow\infty$ and then $\epsilon\rightarrow0,$
where $B_{\epsilon}(\mu)$ denotes a ball of radius $\epsilon$ centered
at a given probability measure $\mu$ in the space $\mathcal{M}_{1}(X)$
(equipped with a metric compatible with the weak topology). In other
words this means that the probability of finding a cloud of $N$ points
$x_{1},...,x_{N}$ on $X$ such that the corresponding measure $\frac{1}{N_{k}}\sum_{i}\delta_{x_{i}}$
approximates a volume form $\mu$ is exponentially small unless $\mu$
is a minimizer of $F.$
\begin{cor}
Let $X$ be a variety of general type. Then the one-point correlation
measures $\nu_{k}:=\int_{X^{N-1}}\mu^{(N_{k})}$ of the canonical
point processes define a sequence of canonical measures on $X$ converging
weakly to $dV_{KE}$. Moreover, the curvature forms of the corresponding
metrics on $K_{X}$ defined by the sequence $\nu_{k}$ converge weakly
to the unique Kähler-Einstein metric $\omega_{KE}$ on $X.$
\end{cor}
The last statement in the previous corollary concretely says that
the unique (possible singular) Kähler-Einstein metric $\omega_{KE}$
on $X$ (whose existence was established in the seminal works of Aubin
\cite{au} and Yau \cite{y} when $K_{X}$ is ample and in \cite{egz,begz,bbgz}
in general) may be recovered as the weak limit of the following sequence
of quasi-explicit canonical positive currents in $c_{1}(K_{X}):$

\begin{equation}
\omega_{k}:=\frac{i}{2\pi}\partial\bar{\partial}\log\int_{X^{N_{k}-1}}\left|(\det S^{(k)})(\cdot,z_{1},...,z_{N_{k}-1})\right|^{2/k}dz_{1}\wedge d\bar{z}_{1}\wedge\cdots\wedge dz_{N_{k}-1}\wedge d\bar{z}_{N_{k}-1}\label{eq:def of canonical sequ of current intro}
\end{equation}
which are smooth away from the base locus of $kK_{X}.$ It may be
worthwhile pointing out that it does not seem clear how to prove the
convergence of the one-point correlation measures without first proving
the full large deviation principle in the previous theorem, which
in general is a considerably stronger type of convergence result.

The convergence of the currents $\omega_{k}$ towards $\omega_{KE}$
is somewhat analogous to Donaldson's convergence of balanced metrics
towards the constant scalar curvature metric attached to a polarized
manifold $(X,L)$ with finite automorphism group \cite{do1}. Indeed,
in both situations the approximating metrics in question are singled
our by the sequence of finite dimensional spaces $H^{0}(X,kL).$ One
virtue of the present setting is that it applies when $L(=K_{X})$
is merely big thus providing what seems to be the first general approximation
result for the singular Kähler-Einstein metric $\omega_{KE}$ on a
variety $X$ of general type.

\subsubsection{Birational invariance and varieties of positive Kodaira dimension}

The results for varieties of general type above generalize to the
setting of a projective variety $X$ of positive Kodaira dimension
$\kappa.$ The starting point is the observation that the canonical
random point processes introduced above are well-defined as long as
the plurigenera $N_{k}$ are non-zero and that they are invariant
under birational equivalence of varieties. In order to do asymptotics
we also need $N_{k}$ to tend to infinity, which means that the natural
setting for the canonical point processes is a projective variety
$X$ of positive Kodaira dimension $\kappa$ (recall that $\kappa$
is the natural number defined by the growth property $N_{k}\sim k^{\kappa}).$ 
\begin{thm}
Let $X$ be a projective variety of positive Kodaira dimension. Then
the empirical measures of the canonical random point processes on
$X$ converge in probability towards a unique probability measure
$\mu_{X},$ which coincides with the canonical measure of Song-Tian
and Tsuji. 
\end{thm}
The canonical measure $\mu_{X}$ was introduced by Song-Tian \cite{s-t}
in their study of the Kähler-Ricci flow and, independently, by Tsuji
\cite{ts} in his study of dynamical systems defined by Bergman kernels.
The measure $\mu_{X}$ is defined in terms of the Iitaka fibration
$F:\,X\dashrightarrow Y$ attached to $X:$ on the Zariski open subset
$Y_{0}$ where $F$ defines a smooth morphism the measure $\mu_{X}$
may be written as 
\[
\mu_{X}=(F^{*}\omega_{Y})^{\kappa}\wedge(\omega_{CY})^{n-\kappa},
\]
 where, fiberwise, $\omega_{CY}$ is a Ricci flat metric with normalized
volume and $\omega_{Y}$ is a canonical metric defined on the base
$Y$ in terms of a twisted Kähler-Einstein equation: it satisfies
(in a weak sense) the equation 
\[
\mbox{Ric }\ensuremath{\omega_{Y}=-\omega_{Y}+\omega_{WP}}
\]
 away from the branch locus of the fibration, where $\omega_{WP}$
is the generalized Weil-Petersson metric on $Y,$ measuring the infinitesimal
variation of the complex moduli of the Calabi-Yau fibers (see section
\ref{sec:Varieties-of-positive} for precise definitions). The proof
of the previous theorem relies on Fujino-Mori's canonical bundle formula
\cite{f-m} which allows one to reduce the problem to the base $Y$
of the Iitaka fibration. A similar argument was used by Tsuji in a
different context \cite{ts}. It was shown by Song-Tian \cite{s-t}
that, in the case when $K_{X}$ is semi-ample, the Kähler-Ricci flow
on $X$ converges weakly towards a canonical positive current which
coincides with $F^{*}\omega_{Y}.$ In our setting we obtain as an
immediate corollary of the previous theorem that, for \emph{any }variety
of positive Kodaira dimension, the canonical sequence of currents
$\omega_{k}\in c_{1}(K_{X})$ defined by formula \ref{eq:def of canonical sequ of current intro}
converges weakly to a canonical positive current in $c_{1}(K_{X})$
with minimal singularities, coinciding with $F^{*}\omega_{Y}$ on
a Zariski open subset of $X.$ This is hence more general then assuming
that $K_{X}$ be semi-ample. On the other hand, assuming the validity
of the fundamental conjectures of the MMP, i.e. the existence of a
minimal model and the abundance conjecture (which is not necessary
for our approach) the convergence of $\omega_{k}$ reduces to the
case when $K_{X}$ is semi-ample.

\subsubsection{Logarithmic generalizations and twisted Kähler-Einstein metrics}

The results stated above admit natural generalization to the logarithmic
setting, in the sense of MMP, which from the differential geometric
point of view is related to Kähler-Einstein metrics with conical and
cusp type singularities. To explain this first recall that (in the
smooth setting) a \emph{log canonical pair} $(X,D)$ consists of a
smooth complex algebraic variety $X$ and a $\Q-$divisor $D$ on
$X$ with simple normal crossings and coefficients in $]-\infty,1].$
In this setting the role of the canonical line bundle $K_{X}$ is
placed by the \emph{log canonical line bundle }$K_{X}+D$ and the
role of the Ricci curvature $\mbox{Ric \ensuremath{\omega}}$of a
metric $\omega$ is played by twisted Ricci curvature $\mbox{Ric \ensuremath{\omega}}-[D],$
where $[D]$ denotes the current of integration defined by $D.$ The
corresponding (twisted) Kähler-Einstein equation thus reads 
\[
\mbox{Ric \ensuremath{\omega}}=\beta\omega+[D]
\]
which should be interpreted in a weak sense (see \cite{ber-gu} for
a very general setting). In case the coefficients of $D$ are in $[0,1[$
the pair $(X,D)$ is said to be\emph{ Kawamata log terminal }(\emph{klt,}
for short) and if $K_{X}+D$ is ample the solution of the corresponding
(twisted) Kähler-Einstein equation then has conical singularities
along $D$ \cite{do-3,jmr,cgh,g-p}. 

To any klt log pair $(X,D)$ we may attach a sequence of canonical
probability measures $\mu^{(N_{k})}$ on $X^{N_{k}}$ defined as follows.
We take $\det S^{(k)}$ to be a generator in the determinant line
of $H^{0}(X,k(K_{X}+D))$ and multiply the local volume forms $dz\wedge d\bar{z}$
with $1/|s_{D}|^{2},$ where $s_{D}$ is a global holomorphic (multivalued)
section cutting out $D.$ The proof of the corresponding convergence
results can then be deduced from the general Theorem \ref{thm:big line bundle text}
, using that the pair $(X,D)$ determines a finite measure $\mu_{0}$
which is absolutely continuous with respect to the Lebesgue measure
(with an $L^{p}-$density, for some $p>1)$ whose singularities/degeneracies
are encoded by the divisor $D.$ 

More generally, as will be shown elsewhere, the convergence results
for klt pairs can be extended to the setting of log canonical pairs
$(X,D).$ In the particular case when the the pair $(X,D)$ is defined
over the integers the present probabilistic framework then turns out
to be connected to the theory of cusp forms and canonical heights
in Arakelov geometry. In a nutshell, the point is that when $(X,D)$
is defined over the integers the corresponding generator $\det S^{(k)}$
can be taken to be defined over the integers and since such a generator
is uniquely determined up to a sign the corresponding partition function
$Z_{k}$ is canonically attached to the integral structure. Using
the generalization of Gillet-Soulé's arithmetic Hilbert-Samuel theorem
in \cite{ber-mont} it can then be shown that $-\frac{1}{N_{k}}\log Z_{N_{k}}$
converges, as $k\rightarrow\infty,$ to the height (i.e. the top arithmetic
intersection number) of the corresponding integral model of $K_{X}+D,$
with respect to the Kähler-Einstein metric.

\subsubsection*{Acknowledgment}

It is a pleasure to thank Sebastien Boucksom, David Witt-Nyström,
Vincent Guedj and Ahmed Zeriahi for the stimulating collaborations
\cite{b-b,b-b-w,bbgz}, which paved the way for the present work.
I am also grateful to Bo Berndtsson for infinitely many fruitful discussions
on complex analysis and Kähler geometry over the years. The present
paper, together with the companion paper \cite{berm8}, supersedes
the first arXiv version of the paper (the latter paper also contained
results about log pairs with log canonical singularities and applications
to arithmetic geometry, which will appear in a separate paper). This
work was supported by grants from the ERC and the KAW foundation

\subsubsection*{Organization}

To make the paper reasonably self-contained we start in Section \ref{sec:Outline-of-the}
by outlining the proof of the LDP in Theorem \ref{thm:thm gen type intro}
and point out some relations to statistical mechanics which are the
subject of the companion paper \cite{berm8}. Then in Section \ref{sec:analytic}
we setup the analytical/pluripotential framework needed for the actual
proofs of the main results. Since the paper has been written mainly
for readers with a background in complex geometry in mind we start
Section \ref{sec:The-LDP-for} with a lightning review of the relevant
notions in probability theory (that can be found in any standard probability
text book). Then the proof of a general LDP for certain $\beta-$deformations
of determinantal point processes is given (using the asymptotic submean
inequality proved in the companion paper \cite{berm8}, where a general
LDP for Gibbs measures is also proved. In Section \ref{sec:Varieties-of-positive}
it is shown how to apply the general LDP in the complex geometric
framework of canonical point processes of klt pairs and varieties
of of positive Kodaira dimension to prove the results formulated in
the introduction of the paper. Finally, in Section \ref{sec:Fano-manifolds-and}
we show how to define canonical point processes on Fano manifolds
and discuss the relations to a new notion of stability, stating a
number of conjectures.

\section{\label{sec:Outline-of-the}Outline of the proof of the convergence
and relations to statistical mechanics}

In this section we outline the proof of the large deviation principle
in Theorem \ref{thm:conv for log general type} concerning the case
when $X$ is a variety of general type, i.e. the canonical line bundle
$K_{X}$ is big. In fact the same arguments apply to the more general
setting of $\beta-$deformed determinantal point processes associated
to a big line bundle $L$ described in Section \ref{sec:The-LDP-for}.

Fixing a reference volume form $dV$ on $X$ and denoting by $\left\Vert \cdot\right\Vert $
the corresponding metric on the canonical line bundle $L:=K_{X}$
the canonical probability measure $\mu^{(N_{k})}$ on $X^{N}$ can
be factorized as 
\[
\mu^{(N_{k})}=\frac{1}{Z_{N_{k}}}\left\Vert \det S^{(k)}\right\Vert ^{2/k}dV^{\otimes N_{k}},
\]
in terms of the induced metric on $K_{X^{N_{k}}}$ for any given generator
$\det S^{(k)}$ of the determinant line $\det H^{0}(X,kK_{X}).$ It
will be convenient to take $\det S^{(k)}$ to be the generator determined
by a basis in $H^{0}(X,kL)$ which is orthonormal with respect to
the $L^{2}-$product determined by $(\left\Vert \cdot\right\Vert ,dV)$
for any fixed volume form $dV$ on $X.$ By general principles (Lemma
\ref{lem:ldp for nonnormalized meas}) it will be enough to obtain
the corresponding LDP for the non-normalized measures on $X^{N_{k}}$
obtained by discarding the normalization factor $Z_{N_{k}}.$ Fixing
a metric on the space $\mathcal{M}_{1}(X)$ of all normalized measures
$\mu$ on $X,$ compatible with the weak topology, it will be convenient
to identify a ball $B_{\epsilon}(\mu)$ in $\mathcal{M}_{1}(X),$
centered at $\mu$ of radius $\epsilon,$ with a subset of $X^{N}$
using the map 
\[
\delta_{N}:=\frac{1}{N}\sum_{i=1}^{N}\delta_{x_{i}}:X^{N}\rightarrow\mathcal{M}_{1}(X).
\]
Then the large deviation principle (LDP) in Theorem \ref{thm:conv for log general type},
is equivalent to the following upper and lower bounds on the integral
of $\mu^{(N_{k})}$ over small balls $B_{\epsilon}(\mu):$ 
\begin{equation}
\lim_{\epsilon\rightarrow0}\limsup_{k\rightarrow\infty}\frac{1}{N_{k}}\log\int_{B_{\epsilon}(\mu)}\left\Vert \det S^{(k)}\right\Vert ^{2/k}dV^{\otimes N_{k}}\leq-F(\mu)\label{eq:ldp in outline}
\end{equation}
and 
\[
\lim_{\epsilon\rightarrow0}\liminf\frac{1}{N_{k}}\log\int_{B_{\epsilon}(\mu)}\left\Vert \det S^{(k)}\right\Vert ^{2/k}dV^{\otimes N_{k}}\geq-F(\mu),
\]
for a functional $F$ on $\mathcal{M}_{1}(X)$ called the\emph{ rate
functional }of the LDP (in the present setting this could be taken
as the definition of the LDP, but see Section \ref{sub:Probabilistic-preliminaries}
for the general situation). The idea of the proof of the LDP is to
handle the the density $\left\Vert \det S^{(k)}\right\Vert ^{2/k}$
and the reference measure $dV^{\otimes N_{k}}$ separately. The point
is that, if the density would have been constant, i.e. if $\mu^{(N)}=dV^{\otimes N},$
a classical theorem of Sanov says that the LDP above holds with rate
function given by the \emph{entropy of $\mu$ relative to $\mu_{0},$
}i.e. the functional defined by 
\[
D_{\mu_{0}}(\mu):=\int_{X}\log\frac{\mu}{\mu_{0}}\mu,
\]
 when $\mu$ has a density with respect to $\mu_{0}$ and otherwise
$D_{\mu_{0}}(\mu):=\infty.$ To handle the present setting we will
combine Sanov's theorem with the logarithmic asymptotics for the sup
of $\left\Vert \det S^{(k)}\right\Vert ^{2}$ established in \cite{b-b}:
for any given any function $u\in C^{0}(X)$ 
\begin{equation}
\lim_{k\rightarrow\infty}\frac{1}{kN_{k}}\log\sup_{X^{N_{k}}}\left(\left\Vert \det S^{(k)}\right\Vert ^{2}(x_{1},....,x_{N_{k}})e^{-ku(x_{1})+...+ku(x_{N_{k}})}\right)=-\mathcal{F}(u),\label{eq:asymptotics of sup in outline}
\end{equation}
 where $\mathcal{F}$ is a certain functional on the space $C^{0}(X)$
whose definition we will come back to below (in geometric terms the
factor involving $u$ just corresponds to replacing the background
metric $\left\Vert \cdot\right\Vert $ with $\left\Vert \cdot\right\Vert _{u}:=\left\Vert \cdot\right\Vert e^{-u/2}).$

\subsection{The upper bound in the LDP}

To prove the upper bound of the integrals appearing in the LDP \ref{eq:ldp in outline}
we fix a function $u\in C_{0}(X)$ and rewrite 
\[
\left\Vert \det S^{(k)}\right\Vert ^{2/k}=\left(\left\Vert \det S^{(k)}\right\Vert ^{2}e^{-ku}\right)^{1/k}e^{u},
\]
 where we have identified $u$ with the corresponding function $u(x_{1})+...+u(x_{N_{k}})$
on $X^{N}.$ Then, trivially, for any fixed $\epsilon>0,$ 
\begin{equation}
\int_{B_{\epsilon}(\mu)}\left\Vert \det S^{(k)}\right\Vert ^{2/k}dV^{\otimes N_{k}}\leq\sup_{B_{\epsilon}(\mu)}\left(\left\Vert \det S^{(k)}\right\Vert ^{2}e^{-ku}\right)^{1/k}\int_{X^{N}}\mu_{u}^{\otimes N_{k}},\,\,\mu_{u}:=e^{u}dV\label{eq:trivial ineq in outline}
\end{equation}
Hence, replacing the sup over $B_{\epsilon}(\mu)$ with the sup over
all of $X^{N_{k}}$and applying Sanov's theorem relative to the tilted
volume form $\mu_{u}$ gives 
\[
\lim_{\epsilon\rightarrow0}\limsup_{k\rightarrow\infty}\frac{1}{N_{k}}\log\int_{B_{\epsilon}(\mu)}\left\Vert \det S^{(k)}\right\Vert ^{2/k}dV^{\otimes N_{k}}\leq-\mathcal{F}(u)-\int_{X}u\mu-D_{dV}(\mu),
\]
 using that $D_{e^{u}dV}(\mu)=-\int u\mu+D_{dV}(\mu).$ Finally, taking
the infimum over all $u\in C^{0}(X)$ gives an upper bound on the
$\limsup$ appearing in formula \ref{eq:ldp in outline} with $F(\mu)$
defined by 
\[
F(\mu):=\mathcal{F}^{*}(\mu)+D_{dV}(\mu),\,\,\,\,\mathcal{F}^{*}(\mu):=\sup_{u\in C^{0}(X)}\left(\int_{X}u\mu+\mathcal{F}(u)\right)
\]

\subsection{The lower bound in the LDP}

As usually the proof of the lower bound in the LDP is the hardest
one. The starting point is the fact, proved in \cite{b-b}, that the
limiting functional $\mathcal{F}$ on $C^{0}(X)$ appearing in formula\ref{eq:asymptotics of sup in outline}
is Gateaux differentiable on $C^{0}(X)$ (it is also concave, as follows
directly from its definition) and its differential at a given $u\in C^{0}(X)$
is represented by a probability measure $d\mathcal{F}_{|u}.$ In particular,
the sup in the formula defining $\mathcal{F}^{*}(\mu)$ is attained
iff there exists a function $u_{\mu}\in C^{0}(X)$ such that 
\begin{equation}
d\mathcal{F}_{|u_{\mu}}=\mu.\label{eq:in image of gradient}
\end{equation}
 Assuming that the measure $\mu$ appearing in the LDP has this property
the idea of the proof of the LDP is to try to reverse the trivial
inequality \ref{eq:trivial ineq in outline} for $u=u_{\mu}.$ To
this end we first recall that, as observed in \cite{b-b-w}, it follows
from a simple convex analysis argument that, if $\mathbf{x}^{(N_{k})}$
in $X^{N_{k}}$ denotes a sequence of ordered configurations of points
realizing the sup in formula \ref{eq:asymptotics of sup in outline}
for a given $u$, then 
\[
\delta_{N}(\mathbf{x}^{(N_{k})})\rightarrow d\mathcal{F}_{|u}
\]
 in $\mathcal{P}(X),$ with respect to the weak topology. In particular,
taking $u=u_{\mu}$ it follows that, for any fixed $\epsilon>0,$
\[
B_{\epsilon/2}(\mu)\subset B_{\epsilon}\left(\delta(\mathbf{x}^{(N_{k})})\right)\subset B_{2\epsilon}(\mu)
\]
 for $N$ sufficiently large. Accordingly, 
\[
\int_{B_{2\epsilon}(\mu)}\left\Vert \det S^{(k)}\right\Vert ^{2/k}dV^{\otimes N_{k}}\geq\int_{B_{\epsilon}\left(\delta(\mathbf{x}^{(N_{k})})\right)}\left\Vert \det S^{(k)}\right\Vert _{u}^{2/k}\mu_{u}^{\otimes N}
\]
and hence it will be enough to establish an asymptotic submean property
of the following form 
\begin{equation}
\int_{B_{\epsilon}\left(\delta(\mathbf{x}^{(N_{k})})\right)}\left\Vert \det S^{(k)}\right\Vert _{u}^{2/k}\mu_{u}^{\otimes N}\geq e^{N\epsilon'}\left\Vert \det S^{(k)}(\mathbf{x}^{(N_{k})})\right\Vert _{u}^{2/k}\int_{B_{\epsilon''}\left(\delta(\mathbf{x}^{(N_{k})})\right)}\mu_{u}^{\otimes N},\label{eq:submean ineq in outline}
\end{equation}
 for some positive numbers $\epsilon'$ and $\epsilon''$ smaller
than $\epsilon.$ Indeed, we would then get 
\[
\liminf_{k\rightarrow\infty}\frac{1}{N_{k}}\log\int_{B_{2\epsilon}(\mu)}\left\Vert \det S^{(k)}\right\Vert ^{2/k}dV^{\otimes N_{k}}\geq-\epsilon'-\mathcal{F}(u_{\mu})+\liminf_{k\rightarrow\infty}\frac{1}{N_{k}}\log\int_{B_{\epsilon''/2}(\mu)}\mu_{u}^{\otimes N}
\]
and letting $\epsilon\rightarrow0$ we could then conclude the proof
by invoking Sanov's theorem again. We recall that we made the assumption
\ref{eq:in image of gradient} that the fixed measure $\mu$ be contained
in the image of $C^{0}(X)$ under the ``gradient map'' $d\mathcal{F}.$
But in fact all that is needed in the proof of the lower bound in
the LDP given above is that any $\mu$ such that $\mathcal{F}^{*}(\mu)<\infty$
has the following ``regularization property'': there exists a sequence
of $u_{j}\in C^{0}(X)$ such that 
\begin{equation}
\mu_{j}:=d\mathcal{F}_{|u_{j}}\rightarrow\mu,\,\,\,\mathcal{F}^{*}(\mu_{j})\rightarrow\mathcal{F}^{*}(\mu)\label{eq:density property}
\end{equation}
In the present setting this property is shown in Lemma \ref{lem:appr}
(but it can also be obtained from general convex analysis results
\cite{berm8}). 

As for the submean property \ref{eq:submean ineq in outline} it is
proved in the companion paper \cite{berm8}, with $\epsilon'=\epsilon$
and $\epsilon''=\epsilon^{2}$ (the result also appeared as Theorem
2.1 in the first preprint version of the present paper on ArXiv).
The starting point of the proof is the observation that if the metric
on $\mathcal{M}_{1}(X)$ is taken as the $L^{2}-$Wasserstein metric
induced by a fixed Kähler metric on $X$ with volume form $dV,$ then
$B_{\epsilon}\left(\delta(\mathbf{x}^{(N_{k})})\right)$ may be identified
with the pullback to $X^{N}$ of the ball of radius $\epsilon$ centered
at the projection of the point $\mathbf{x}^{(N_{k})}$ in the Riemannian
orbifold $X^{(N)}:=X^{N}/S^{N},$ where $S^{N}$ denotes the symmetric
group with $N$ elements acting by permutations on $X^{N}.$ The inequality
in question then follows from a new submean inequality for quasi-subharmonic
functions on Riemannian orbifold quotients in large dimensions, obtained
as refinement of a Riemannian submean inequality of Li-Schoen \cite{li-sc}
(using, in particular, the explicit dependence on the dimension in
the Cheng-Yau gradient estimate on Riemannian manifolds).

\subsection{Identification of the rate function with the K-energy}

It follows from general principles that the rate functional $F(\mu)$
is lower semi-continuous (lsc) and strictly convex on $\mathcal{M}_{1}(X)$
(using the strict convexity of the relative entropy). In particular,
it admits a unique minimizer $\mu_{min}.$ To see that, in the present
setting, $\mu_{min}$ is the volume form of the Kähler-Einstein metric,
one needs to invoke the explicit expression for the differential $d\mathcal{F}$
obtained in \cite{b-b}:
\[
d\mathcal{F}_{|u}=\frac{1}{V}(\theta+dd^{c}(P_{\theta}u)^{n},\,\,\,\,\,dd^{c}:=\frac{i}{2\pi}\partial\bar{\partial,}\,\,\,V:=K_{X}^{n}
\]
 where $\theta$ is the curvature form of the reference metric $\left\Vert \cdot\right\Vert $
on $K_{X}$ and $P_{\theta}$ is the maximally increasing projection
operator from $C^{0}(X)$ to the space $PSH(X,\theta)$ of all $\theta-$plurisubharmonic
(psh) functions on $X$ and where the complex Monge-Ampère operator
$(\theta+dd^{c}\varphi)^{n}$ on $PSH(X,\theta)$ is defined in the
sense of pluripotential theory (as recalled in Section \ref{sec:analytic}).
In case $K_{X}$ is ample one can then invoke the Aubin-Yau theorem
giving the existence of a Kähler-Einstein metric $\omega_{KE}$ on
$X$ and directly check that its volume form $\mu_{KE}$ is a critical
point of $\mathcal{F}$ and hence, by convexity, $\mu_{KE}=\mu_{min}$
(see \cite{berm8}). Alternatively, under the Calabi-Yau isomorphism
\[
\omega\mapsto\omega^{n}/V
\]
between the space $\mathcal{K}(K_{X})$ of all Kähler metrics in $K_{X}$
and the subspace of $\mathcal{M}_{1}(X)$ consisting of volume forms
the rate functional $F$ gets identified (on the dense subset of volume
forms) with Mabuchi's K-energy functional $\kappa$ whose unique critical
point in $\mathcal{K}(K_{X})$ is the Kähler-Einstein metric $\omega_{KE}$
(see Section \ref{sub:Identification-of-the}). One virtue of the
functional $F$ is that it is defined on the whole space $\mathcal{M}_{1}(X)$
and can thus be used as an extension of the K-energy to the space
of all metrics of finite energy \cite{berm6,bbgez} (as further exploited
in \cite{bdl} to study a weak version of the Calabi flow realized
as a weak gradient flow of the extended K-energy functional, following
\cite{st}).

In the general case of a variety $X$ of general type the existence
of a unique Kähler-Einstein current $\omega_{KE}$ in $c_{1}(K_{X})$
with minimal singularities was established in \cite{begz} (using
various approximations to reduce the problem to the Aubin-Yau theorem
on a sequence of blow-ups of $X)$ and in \cite{bbgz} by a direct
variational approach. The current $\omega_{KE}$ can also be obtained
as a consequence of the results in \cite{egz} concerning Kähler-Einstein
metrics on canonically polarized varieties with canonical singularities,
using the deep finite generation of the canonical ring \cite{bi}.

\subsection{\label{sub:Relations-to-statistical}Relations to statistical mechanics
and Gibbs measures}

Next, we briefly point out some relations between the present setting
and statistical mechanics (which are further expanded on in \cite{berm8}).
Consider a system of $N$ identical particles on a space $X$ interacting
by the microscopic interaction energy $H^{(N)}(x_{1},...,x_{N})$
on $X^{N}$ assumed symmetric (as the particles are identical) - the
function $H^{(N)}$ is often called the \emph{Hamiltonian}. Given
a reference volume form $\mu_{0}$ on $X$ the distribution of particles,
in thermal equilibrium at the inverse temperature $\beta\in\R,$ is
described by the random point process with $N$ particles defined
by the corresponding \emph{Gibbs measure:} 
\[
\mu_{\beta}^{(N)}:=e^{-\beta H^{(N)}}\mu_{0}^{\otimes N}/Z_{N,\beta},\,\,\,Z_{N.\beta}:=\int_{X^{N}}e^{-\beta_{N}H^{(N)}}\mu_{0}^{\otimes N},
\]
 where the normalizing constant $Z_{N,\beta}$ is called the\emph{
partition function} and is assumed to be finite. Formally, in the
zero temperature limit $\beta=\infty$ the corresponding empirical
measure $\delta_{N}$ is thus concentrated on the configurations of
points $\mathbf{x}^{(N)}:=(x_{1},...,x_{N})$ minimizing $H^{(N)}$
on $X^{N}.$ If the weak limit 
\[
\lim_{N\rightarrow\infty}\delta_{N}(\mathbf{x}^{(N)}):=\mu_{\infty}
\]
 exists in $\mathcal{M}_{1}(X)$ we can thus think of $\mu_{\infty}$
as representing the deterministic macroscopic equilibrium state of
the system at zero temperature (which is independent of the reference
measure $\mu_{0}$ on $X).$ In the opposite case of infinite temperature,
i.e. $\beta=0$ the system is completely random at the microscopic
level, i.e. the positions $x_{i}$ are independent random points in
$X$ with identical distribution $\mu_{0}.$ However, macroscopically
(i.e. as $N\rightarrow\infty)$ the corresponding random measures
$\delta_{N}$ converge, by Sanov's theorem, exponentially in probability
towards the deterministic measure $\mu_{0}$ at a rate $N$ with rate
function given by the entropy $D_{\mu_{0}}(\mu)$ relative to $\mu_{0}.$
\footnote{The entropy is usually defined as $-D_{\mu_{0}}(\mu)$ in the physics
literature, denoted by $S,$ and Sanov's theorem can be seen as a
mathematical justification of Boltzmann's original formula expressing
the entropy as the logarithm of the number of microscopic states consistent
with a given macroscopic state.} But the question is what happens in at a fixed finite positive inverse
temperature $\beta?$ According, to some time-honored heuristics in
thermodynamics the corresponding macroscopic measure $\mu_{\beta}$
should minimize the corresponding \emph{free energy functional} $F_{\beta}$
on $\mathcal{M}_{1}(X):$ 
\[
F_{\beta}(\mu):=E(\mu)+\frac{1}{\beta}D_{\mu_{0}}(\mu),
\]
 where $E(\mu)$ is a macroscopic analog of the microscopic interaction
energy $H^{(N)},$ which is minimized on the zero temperature state
$\mu_{\infty.}$ The argument outlined in the previous section (and
proved in the companion paper \cite{berm8} in full generality) shows
that these heuristic can be made precise, in the sense of large deviations,
under the following assumptions on $H^{(N)}$ (assuming for simplicity
that $X$ is compact):
\begin{itemize}
\item The limit of linearly perturbed minimal energies 
\[
\mathcal{F}(u):=\lim_{N\rightarrow\infty}\inf_{X^{N}}\frac{1}{N}(H^{(N)}+u)
\]
exists, for any $u\in C^{0}(X)$ and the functional $\mathcal{F}$
is Gateux differentiable on $C^{0}(X)$ (with a differential $d\mathcal{F}$
taking values in $\mathcal{M}_{1}(X))$
\item The interaction energy $H^{(N)}$ is uniformly quasi-superharmonic,
i.e. there exists a constant $C>0$ such that
\end{itemize}
\[
\Delta_{x_{1}}H^{(N)}\leq C,
\]
where $\Delta$ denotes the Laplacian defined with respect to a fixed
Riemannian metric on $X$ and the subscript $x_{1}$ indicates that
the Laplacian acts on the first variable.

Then the macroscopic energy functional $E(\mu)$ is defined as the
Legendre-Fenchel transform of the functional $\mathcal{F}(u),$ and
the LDP shown in the companion paper \cite{berm8} says that the laws
of the corresponding empirical measures $\delta_{N}$ satisfy a LDP
with rate functional $F_{\beta}$ at a speed $\beta N.$ Moreover,
under suitable regularity assumptions the unique minimizer $\mu_{\beta}$
of $F_{\beta}$ can be represented as $\mu_{\beta}=d\mathcal{F}_{|u_{\beta}}$
for a function $u_{\beta}$ satisfying the following equation: 
\[
d\mathcal{F}_{|u_{\beta}}=e^{\beta u_{\beta}}
\]
However, unless $H^{(N)}$ is uniformly equicontinuous the function
$u_{\beta}$ will, in general, not be continuous. For example, in
the present setting $u_{\beta}$ will be a $\theta-$psh function
(of finite energy).

The general LDP above generalizes (in the case \textbf{$\beta>0)$}
the mean field type results in \cite{clmp,k} concerning the vortex
model for turbulence in two real dimensions (which in turn extend
to a singular setting previous results in \cite{m-s} ). In the latter
setting the Hamiltonian $H^{(N)}$ is of the explicit form 
\[
H^{(N)}(x_{1},....,x_{N})=-\frac{1}{(N-1)}\sum_{1\leq i<j\leq N}G(x_{i},x_{j})
\]
for a symmetric function $g$ (independent of $N)$ - in fact, $G$
is the Green function of the corresponding Laplacian, which is thus
singular along the diagonal. However, while the analysis in \cite{clmp,k}
reduces, thanks to the explicit formula above, to properties of the
function $G,$ the main point of the present approach is that it applies
in situations where $H^{(N)}(x_{1},....,x_{N})$ does not admit any
tractable explicit formula, as in the complex geometric setting. In
fact, a tractable formula does exist in\emph{ one} complex dimension,
namely the so called \emph{bosonization formula} on a Riemann surface
which involves explicit theta functions and regularized determinants
of Laplacians (the formula was used in a related large deviation setting
in \cite{z2}). The bosonization formula is particularly useful in
the case $\beta<0$ on a Riemann surface as will be explained in a
separate publication.

\section{\label{sec:analytic}Analytic setup }

In this section we introduce the analytical framework from \cite{b-b,bbgz}
needed for the proofs of the main results stated in the introduction
of the paper. As a courtesy to the reader we also sketch the proofs
of the relevant results in \cite{b-b,bbgz}.

\subsection{\label{sub:Setup complex}Setup}

Let $L\rightarrow X$ be a holomorphic line bundle over an $n-$dimensional
compact complex manifold $X.$

\subsubsection{Metrics, weights and $\theta-$psh functions}

We will denote by $\left\Vert \cdot\right\Vert $ a Hermitian metric
on $L.$ Occasionally we will use additive notation for metrics, where
the metric is represented by a \emph{weight} denoted by $\phi,$ which
is a short hand for a collection of local functions: if $s$ is a
trivializing local holomorphic section of $L,$ i.e. $s$ is non-vanishing
an a given open set $U$ in $X,$ then $\phi_{|U}:=\log\left\Vert s\right\Vert ^{2}.$
In this notation the normalized curvature form of $\left\Vert \cdot\right\Vert $
may be (locally) written as 
\[
\theta:=dd^{c}\phi,\,\,\,\,dd^{c}:=\frac{i}{2\pi}\partial\bar{\partial,}
\]
representing, globally, the first Chern class $c_{1}(L)$ in the cohomology
group $H^{2}(X,\R)\cap H^{2}(X,\Z).$ More generally, if the weight
$\phi$ is in $L_{loc}^{1},$ the previous formula defines the curvature
as a current on $X.$ A (possibly singular( metric is said to be of
positive curvature if $\theta\geq0$ holds in the sense of currents.
Following standard practice it will also be convenient to identify
positively curved metrics on $L$ with $\theta-$psh functions, as
follows. Fixing once and for all a smooth Hermitian metric $\left\Vert \cdot\right\Vert $
on $L$ any other continuous metric may be written as 
\[
\left\Vert \cdot\right\Vert _{\varphi}^{2}:=e^{-\varphi}\left\Vert \cdot\right\Vert ^{2}
\]
for a continuous function $\varphi$ on $X,$ i.e. $\varphi\in C^{0}(X).$
In particular, the curvature current of the metric $\left\Vert \cdot\right\Vert _{\varphi}^{2}$
may then be written as 
\[
\theta_{\varphi}:=\theta+dd^{c}\varphi
\]
More generally, this procedure gives a correspondence between the
space of all (singular) metrics on $L$ with positive curvature current
and the space $PSH(X,\theta)$ defined as the space of of all upper-semi
continuous functions $\varphi$ on $X$ such that $\theta_{\varphi}\geq0$
holds in the sense of currents. We will use the same notation $\left\Vert \cdot\right\Vert $
for the induced metrics on tensor powers of $L$ etc. In the weight
notation the induced weight on the $k$ th tensor power $kL$ is thus
given by $k\phi.$ We recall that $\varphi_{1}\in PSH(X,\theta)$
is said to be more singular than $\varphi_{2}\in PSH(X,\theta)$ if
there exists a constant $C$ such that $\varphi_{1}\leq\varphi_{2}+C.$
\begin{example}
\label{exa:Ricci curvature}A volume form $dV$ on $X$ induces a
smooth metric on $L:=K_{X}:=\det(T^{*}X)$ whose local weight is given
by $\phi_{U}:=\log\frac{dV}{\frac{i}{2}dz_{1}\wedge d\bar{z}_{1}\wedge\cdots}$
in the local trivialization $\frac{i}{2}dz_{1}\wedge d\bar{z}_{1}\wedge\cdots\wedge\frac{i}{2}dz_{n}\wedge d\bar{z}_{n}$
of $K_{X}$ induced by given local holomorphic coordinates $z_{1},...,z_{n}$
on $U\subset X.$ If $dV$ is the volume form of a Kähler metric $\omega$
on $X,$ then the curvature form $\theta$ of the corresponding metric
on $K_{X}$ coincides with $-\mbox{Ric}\,\omega.$ 
\end{example}

\subsubsection{Holomorphic sections and big line bundles}

We will denote by $H^{0}(X,kL)$ the space of all global holomorphic
sections with values in the $k$ th tensor power of $L.$ We will
usually assume that $L$ is\emph{ big}, i.e. 
\[
N_{k}:=\dim H^{0}(X,kL)=Vk^{n}+o(k^{n}),
\]
 for a positive number $V,$ called the\emph{ volume of $L.$ }In
the case when $L=K_{X}$ and $s_{k}\in H^{0}(X,kK_{X})$ we can attach
a canonical measure $\mu_{s_{k}}$ on $X$ to $s_{k}$ which we will,
abusing notation somewhat, sometimes write as 
\[
\mu_{s_{k}}:=i^{n^{2}}\left(s_{k}\wedge\overline{s_{k}}\right)^{1/k}=\left|s_{k}\right|^{2/k}idz\wedge d\bar{z},
\]
 where in the right hand side we have locally identified $s_{k}$
with a holomorphic function, defined with respect to the local trivialization
$dz^{\otimes k}$ of $K_{X}$ ($dz:=dz_{1}\wedge\cdots\wedge dz_{n}$
) induced by a choice of local coordinates $z.$

\subsubsection{Divisors, log pairs and singular volume forms}

We recall that an ($\R-$) divisor $D$ on a complex manifold $X$
is a formal finite sum of one-dimensional irreducible subvarieties:
\[
D=\sum_{i=1}^{m}c_{i}D_{i},\,\,\,\,c_{i}\in\R
\]
Following the standard notation in Minimal Model Program (MMP) a pair
$(X,D)$ is called a \emph{log pair} and $(X,D)$ it is said to be\emph{
log smooth} if $D$ has simple normal crossings (snc), i.e. locally
we can always choose holomorphic coordinates so that $D_{i}=\{z_{i}=0\}).$
Henceforth $(X,D)$ will always refer to a log smooth pair (anyway
this can always be arranged by passing to a log resolution). The pair
$(X,D)$ is said to be\emph{ log canonical (lc)} if $c_{i}\leq1$
and\emph{ Kawamata Log Terminal (klt)} if $c_{i}<1.$ Denoting by
$s_{i}$ a holomorphic section of the line bundle $\mathcal{O}(D_{i})$
cutting out $D_{i}$ we will use the symbolic notation $s_{D}:=s_{1}^{c_{1}}\cdots s_{m}^{c_{m}}$
(which can viewed as a multi-section when $c_{i}\in\Q).$ The point
is that $\phi_{D}:=\log|s_{D}|^{2}$ is then a well-defined weight
on the corresponding $\R-$line bundle $\mathcal{O}(D)$ and its curvature
current coincides with the integration current $[D]$ defined by the
divisor $D.$ Given a volume form $dV$ on $X$ and a continuous metric
on the $\R-$line bundle $\mathcal{O}(D)$ we obtain a measure 

\[
\mu=\left\Vert s_{D}\right\Vert ^{-2}dV
\]
on $X,$ which thus has zeroes and poles along the irreducible divisors
$D_{i}$ with negative and positive coefficients, respectively The
klt assumption is equivalent to demanding that $\mu$ be a finite
measure. We will say that $\mu$ defined as above is a\emph{ singular
volume form in the singularity class defined by $D.$}

\subsection{Preliminaries on complex Monge-Ampère equations and pluripotential
theory}

In this section we will recall some definition and results of global
pluripotential theory in \cite{begz,bbgz,b-b,b-b-w,berm6}.

\subsubsection{The complex Monge-Ampère operator and the pluricomplex energy}

Let $L\rightarrow X$ be a big line bundle and fix, as above, a smooth
Hermitian metric $\left\Vert \cdot\right\Vert $ on $L$ with curvature
form $\theta.$ For a smooth curvature form $\theta_{\varphi}(:=\theta+dd^{c}\varphi)$
one defines the Monge-Ampère measure of the function $\varphi$ (with
respect to $\theta)$ as 
\begin{equation}
MA(\varphi):=\theta_{\varphi}^{n}/V\label{eq:def of ma meas text}
\end{equation}
 i.e. the top exterior power of the corresponding curvature form divided
by the volume $V$ of the class $[\theta]$ (see below). More generally,
according to the classical local pluripotential theory of Bedford-Taylor
the expression in formula \ref{eq:def of ma meas text} makes locally
sense for any bounded $\theta-$psh function (and the corresponding
Monge-Ampère measure does not charge pluripolar subsets, i.e. sets
of the form $\left\{ \varphi=-\infty\right\} $ for a $\theta-$psh
function $\varphi.$ The corresponding local Monge-Ampère measure
is continuous (in the weak topology) with respect to decreasing sequences
of bounded $\theta-$psh functions. In general, following \cite{begz},
for any $\varphi\in PSH(X,\theta)$ we will denote by $MA(\varphi)$
the \emph{non-pluripolar Monge-Ampère measure} $MA(\varphi),$ which
is a globally well-defined measure on $X$ not charging pluripolar
subsets and in particular not Zariski closed set. The measure $MA(\varphi)$
is defined by replacing the ordinary wedge products with the so called
non-pluripolar products introduced in \cite{begz}). In particular,
we recall that for $\varphi$ an element in the subspace $PSH(X,\theta)_{min}$
of functions with minimal singularities the total mass of $\theta_{\varphi}^{n}$
is independent of $\varphi$ and may be taken as the definition of
the volume $V$ of the class $[\theta].$ With the normalization in
\ref{eq:def of ma meas text} this thus means that for any $\varphi$
in $PSH(X,\theta)_{min}$ the measure $MA(\varphi)$ is a probability
measure on $X.$ On the space $PSH(X,\theta)_{min}$ there is an energy
type functional, denoted by $\mathcal{E},$ which may be defined as
a primitive for the one-form defined by $MA(\varphi)/V,$ i.e. 
\begin{equation}
d\mathcal{E}_{|\varphi}=MA(\varphi),\label{eq:def of energyfunc as primitive}
\end{equation}
(in the sense that $d\mathcal{E}(\varphi+tu)/dt=\int_{X}MA(\varphi)u/V$
at $t=0).$ The functional $\mathcal{E}$ is only defined up to an
additive constant which may be fixed by the normalization condition
$\mathcal{E}(v_{\theta})=0$ for some reference element $v_{\theta}$
in $PSH(X,\theta)_{min}$. When $L$ the reference curvature form
$\theta$ can be taken to be a Kähler form. Integrating the defining
relation \ref{eq:def of energyfunc as primitive} along a line segment
in $PSH(X,\theta)_{min}$ reveals that 
\[
\mathcal{E}(\varphi)=\frac{1}{(n+1)V}\int_{X}\varphi\sum_{j=0}^{n}\theta_{\varphi}^{j}\wedge\theta_{\varphi}^{n-j}
\]
(but the explicit formula for $\mathcal{E}$ will not really be used
in the sequel). Occasionally, we will write $\mathcal{E}_{\theta}$
to indicate the dependence of $\mathcal{E}$ on the fixed normalization.
We will also denote by $\mathcal{E}_{\theta}$ the unique upper semi-continuous
extension of $\mathcal{E}_{\theta}$ to all of $PSH(X,\theta)$ and
write 
\[
\mathcal{E}^{1}(X,\theta):=\left\{ \varphi\in PSH(X,\theta):\,\,\mathcal{E}_{\theta}(\varphi)>-\infty\right\} ,
\]
 which is called the space of all functions on $X$ with \emph{finite
energy. }Now following \cite{bbgz} the\emph{ pluri-complex energy}
$E_{\theta}(\mu)$ of a probability measure $\mu$ is defined by
\begin{equation}
E_{\theta}(\mu):=\sup_{\varphi\in PSH(X,\theta)}\mathcal{E}_{\theta}(\varphi)-\left\langle \varphi,\mu\right\rangle ,\label{eq:def of e as sup}
\end{equation}
As recalled in the following theorem the sup defining $E_{\theta}$
is in fact attained: 
\begin{equation}
E_{\theta}(\mu):=\mathcal{E}_{\theta}(\varphi_{\mu})-\left\langle \varphi_{\mu},\mu\right\rangle \label{eq:energy in terms of pot}
\end{equation}
for a unique function $\varphi_{\mu}\in\mathcal{E}^{1}(X,\theta)/\R$
if $E_{\theta}(\mu)<\infty$ where 
\begin{equation}
MA(\varphi_{\mu})=\mu.\label{eq:potential of meas}
\end{equation}

\begin{thm}
\label{thm:var sol of ma}\cite{bbgz} The following is equivalent
for a probability measure $\mu$ on $X:$ 
\begin{itemize}
\item $E_{\theta}(\mu)<\infty$
\item \textup{$\left\langle \varphi,\mu\right\rangle <\infty$ for all $\varphi\in\mathcal{E}^{1}(X,\theta)$}
\item $\mu$ has a potential $\varphi_{\mu}\in\mathcal{E}^{1}(X,\theta$),
i.e. equation \ref{eq:potential of meas} holds
\end{itemize}

Moreover, $\varphi_{\mu}$ is uniquely determined mod $\R,$ i.e.
up to an additive constant and can be characterized as the function
maximizing the functional whose sup defines $E_{\theta}(\mu)$ (formula
\ref{eq:def of e as sup}). 

\end{thm}
\begin{proof}
Since the proof fits naturally into the present probabilistic framework
we outline its main ingredients. The function $\varphi$ will be obtained
by maximizing the following functional on $\mathcal{E}^{1}(X):$
\[
\mathcal{G}(\varphi):=\mathcal{E}(\varphi)-\left\langle \varphi,\mu\right\rangle 
\]
By definition any critical point of $\mathcal{G}$ i.e. function $\varphi$
such that $d\mathcal{G}_{|\varphi}=0$ satisfies the previous equation.
As usual in the direct method of the calculus of variations there
are two main steps:
\begin{itemize}
\item The functional $\mathcal{G}$ admits a minimizer $\psi$ on $\mathcal{E}^{1}$ 
\item The minimizer $\psi$ is a critical point.
\end{itemize}
The last step would be automatic if $\psi$ were known to be an interior
point in the space $\mathcal{E}^{1}.$ However, as the latter space
is only convex $\psi$ could be in the boundary of $\mathcal{E}^{1}$
viewed as a subset of $PSH(X,\theta).$ To get around this difficulty
one observes that, by monotonicity, $\psi$ is also a maximizer of
the functional $\tilde{\mathcal{G}}$ obtained by replacing $\mathcal{E}$
with the functional $\mathcal{F}_{\theta}:=(\mathcal{E}_{\theta}\circ P_{\theta}),$
where $P_{\theta}$ is the projection operator defined below (and
which already appeared in the outline in Section \ref{sec:Outline-of-the}).
The latter functional has the virtue that it is defined on the affine
space $\{\psi\}+C^{0}(X)$ and that it is Gateaux differentiable,
with a differential given by formula in Theorem \ref{thm:thm A and B in b-b}
below. Hence, $\tilde{d\mathcal{G}_{|\psi}}=0,$ which is equivalent
to $d\mathcal{G}_{|\psi}=0,$ since $P_{\theta}\psi=\psi.$ 

Coming back to the first step the assumption $E(\mu)<\infty$ means
precisely that the sup of $\mathcal{G}$ is finite. To find a minimizer
the starting point is the compactness of the space $PSH(X,\theta)_{0}$
of all sup-normalized functions $\varphi$ (i.e. $\sup_{X}\varphi=0),$
using the usual $L^{1}-$topology (which is equivalent to the $L^{p}-$topology
for any $p\geq1).$ Since $\mathcal{G}$ descends to $PSH(X,\theta)_{0}$
and the functional $\mathcal{E}$ is lower semi-continuous all that
remains is to make sure that the integration pairing $\left\langle \cdot,\mu\right\rangle $
is continuous on $\mathcal{E}^{1}(X).$ This follows directly from
the compactness of $PSH(X,\theta)_{0}$ when $\mu$ is a volume form
or more generally when $\mu$ has an $L^{p}-$density for some $p>1.$
However, for a general $\mu$ of finite energy the proof turns out
to be rather subtle. Briefly, one first establishes a general coercivity
type inequality of the form 
\[
|-\left\langle \varphi,\mu\right\rangle |\leq C_{\mu}(-\mathcal{E}(\varphi))^{1/2}
\]
It would then be enough to know that that the integration pairing
$\left\langle \cdot,\mu\right\rangle $ is continuous on any sublevel
of the functional $-\mathcal{E}(\varphi).$ Indeed, taking a sup-normalized
maximizing sequence $\psi_{j}$ (i.e. tending to the sup $E(\mu)$
of $\mathcal{G})$ the coercivity inequality above ensures that $\psi_{j}$
stays in a fixed sublevel set of $-\mathcal{E}(\varphi)$ and hence,
by the lower semi-continuity of $\mathcal{G}$ on such a sublevel
set the $L^{1}-$limit in $PSH(X,\theta)$ is in $\mathcal{E}^{1}$
and maximizes $\mathcal{G},$ as desired. The continuity property
of $\left\langle \cdot,\mu\right\rangle $ in question is indeed shown
in \cite{bbgz}. However the proof is rather indirect and tied up
with the proof of the existence of a finite energy minimizer.\end{proof}
\begin{example}
In the classical case $n=1,$ i.e. $X$ is a Riemann surface and $\theta$
is smooth and strictly positive the space $\mathcal{E}^{1}(X,\theta)$
is the intersection of $PSH(X,\theta)$ with the vector space of all
functions $\varphi$ whose gradient is in $L^{2}(X),$ i.e. $\int d\varphi\wedge d^{c}\varphi<\infty.$
Accordingly, using integration by parts, $E_{\theta}(\mu)=\int d\varphi_{\mu}\wedge d^{c}\varphi_{\mu}/2.$
This means that $E_{\theta}(\mu)$ is the classical Dirichlet energy
of a charge distribution $\mu$ in the ``neutralizing back-ground
charge $\theta"$ (compare \cite{berm 1 komma 5}). 
\end{example}
In the case when the class $[\theta]$ is Kähler and $\mu$ is a volume
form the existence of a smooth solution to equation \ref{eq:potential of meas}
was first shown by Yau \cite{y} in his celebrated solution of the
Calabi conjecture (the uniqueness of such solutions is due to Calabi).

\subsubsection{The psh-projection $P_{\theta},$ the functional $\mathcal{F}_{\theta}$
and asymptotics}

The ``psh-projection'' is the operator $P_{\theta}$ from $C^{0}(X)$
to $PSH(X,\theta)_{min}$ defined as the following envelope:

\begin{equation}
(P_{\theta}u)(x):=\sup\{\varphi(x):\,\,\,\varphi\leq u,\,\,\,\}\label{eq:def of proj operator in khler case}
\end{equation}
Using the latter projection operator it will be convenient to take
the reference element $v_{\theta}$ in $PSH(X,\theta)_{min},$ referred
to above, to be defined by 
\[
v_{\theta}:=P_{\theta}0
\]
We may then define the following functional on $C^{0}(X):$ 
\[
\mathcal{F}_{\theta}(u):=(\mathcal{E}_{\theta}\circ P_{\theta})(u)
\]
Using the latter functional the pluricomplex energy $E_{\theta},$
defined above, may be realized as a Legendre transform:
\begin{prop}
\label{prop:pluri energy as legendre}The pluricomplex energy $E_{\theta}$
is the \emph{Legendre-Fenchel transform} of the convex functional
$u\mapsto f(u):=-(\mathcal{E}_{\theta}\circ P_{\theta})(-u),$ i.e.
\begin{equation}
E_{\theta}(\mu):=\sup_{u\in C^{0}(X)}\mathcal{E}_{\theta}(P_{\theta}u)-\left\langle u,\mu\right\rangle ,\label{eq:E as Leg}
\end{equation}
Moreover
\[
E_{\theta}(\mu)\geq0
\]
 with equality precisely for $\mu:=MA(v_{\theta}),$ where $v_{\theta}:=P_{\theta}0.$ \end{prop}
\begin{proof}
This was shown in \cite{begz} in the ample case and in \cite{berm 1 komma 5}
in the general big case. Briefly, denoting by $E'_{\theta}(\mu)$
the sup appearing in the right hand side of formula \ref{eq:E as Leg}
any $\theta-$psh function of the form $\varphi:=P_{\theta}u$ is
a contender for the sup defining $E_{\theta}(\mu)$ and hence, since
$P_{\theta}$ is decreasing, $E_{\theta}(\mu)\text{\ensuremath{\geq}}E'_{\theta}(\mu).$
To prove the converse first assume that $\mu$ is of the form $\mu=MA(P_{\theta}u)$
for some $u\in C^{0}(X).$ Then the equality $E_{\theta}(\mu)\text{=}E'_{\theta}(\mu)$
follows immediately from the general ``orthogonality relation''
: 
\begin{equation}
\int(u-P_{\theta}u)MA(P_{\theta}u)=0\label{eq:og relation}
\end{equation}
 (saying that $MA(P_{\theta}u)$ is supported on the set $\{P_{\theta}u<u\}$
\cite{b-b}). The case of a general $\mu$ is then proved by approximation
\cite{berm 1 komma 5}. Note that the orthogonality relation also
implies that $E_{\theta}(MA(v_{\theta}))=0.$\end{proof}
\begin{rem}
\label{rem:=00005BLegendre-transform-in} In general, the Legendre-Fenchel
transform of a function $f$ on a topological vector spaces $V$ is
the convex lsc function $f^{*}$ on the topological dual $V^{*}$
defined by 
\[
f^{*}(w):=\sup_{v\in V}\left\langle v,w\right\rangle -f(v)
\]
in terms of the canonical pairing between $V$ and $V^{*}.$ In the
present setting $V=C^{0}(X)$ and $V^{*}=\mathcal{M}(X),$ the space
of all signed Borel measures on $X$ (see for example \cite{d-z}).
\end{rem}
As explained in Section \ref{sec:Outline-of-the} a key ingredient
in the proof of the large deviation principles described in the introduction
of the paper is the following result concerning the existence and
differentiability of the transfinite diameter associated to a big
Hermitian line bundle (which is equivalent to Theorem A and B in \cite{b-b}). 
\begin{thm}
\label{thm:thm A and B in b-b} \cite{b-b}. Let $L\rightarrow X$
be a big line bundle equipped with a continuous Hermitian metric $\left\Vert \cdot\right\Vert $
on $L$ with curvature current $\theta.$ Then 
\begin{itemize}
\item \emph{If $\det S^{(k)}$ denotes the element in the determinant line
of $H^{0}(X,kL)$ induced by a basis in $H^{0}(X,kL)$ which is orthonormal
with respect to the $L^{2}-$norm determined by }$\left\Vert \cdot\right\Vert $
and a volume form $dV$ on $X$ (or more generally a measure $\mu_{0}$
which has the Bernstein-Markov property) then\emph{ 
\[
\lim_{k\rightarrow\infty}-\frac{1}{kN_{k}}\sup_{X^{N}}\left(\log\left\Vert \det S^{(k)}\right\Vert ^{2}(x_{1},...,x_{N})+ku(x_{1})+\cdots+ku(x_{N})\right)=\mathcal{F}_{\theta}(u),
\]
}
\item \emph{The functional $\mathcal{F}_{\theta}$ is Gateaux differentiable
on $C^{0}(X)$ with differential
\begin{equation}
(d\mathcal{F}_{\theta})_{|u}=MA(P_{\theta}u).\label{eq:formula for diff}
\end{equation}
}
\end{itemize}
\end{thm}
\begin{proof}
As a courtesy to the reader we outline the proof of the theorem, starting
with the case when $L$ is ample (compare the end of Section 4 in
\cite{b-b}). Fix a volume form $dV$ on $X.$ Given a continuous
metric $\phi$ on $L$ we set \emph{
\[
\mathcal{F}_{k,L^{p}}[\phi]:=-\frac{1}{kN_{k}}\log\left\Vert \det S^{(k)}\right\Vert _{L^{p}(X^{N_{k}},k\phi,dV)}^{2}
\]
}defined in terms of $L^{p}-$norm on $H^{0}(X^{N},(kL)^{N_{k}})$
induced by $(\phi,dV),$ for $p\in[1,\infty]$ (which is defined to
be the ordinary sup-norm for $p=\infty$ and thus independent of $dV).$
For $p=\infty,$ which is the case appearing in the statement of the
theorem, it follows immediately from the definitions that 
\begin{equation}
\mathcal{F}_{k,L^{\infty}}[\phi]=\mathcal{F}_{k,L^{\infty}}[P\phi]\label{eq:norm wrt phi as wrt P phi}
\end{equation}
and that $\mathcal{F}_{k,L^{\infty}}$ is equicontinuous with respect
to the sup-norm on the space of continuous metrics. Accordingly, writing
the positively curved metric $P\phi$ as the uniform limit of smooth
and positively curved metrics $\psi_{j}$ (using Demailly's approximation
theorem on an ample line) it is enough to prove the convergence of
$\mathcal{F}_{k,L^{\infty}}[\psi]$ when $\psi$ is a smooth metric
with strictly positive curvature. To this end one uses that 
\begin{equation}
\mathcal{F}_{k,L^{\infty}}[\psi]=\mathcal{F}_{k,L^{2}}[\psi]+o(1),\label{eq:distortion of norms}
\end{equation}
where the error term $o(1)$ (tending to zero) only depends on the
modulus of continuity of $\psi.$ Indeed, this follows directly from
applying the standard submean property of holomorphic functions on
small coordinate balls on $X,$ for each factor of $X^{N_{k}}.$ Now,
a direct calculation reveals that the differential of $\mathcal{F}_{k,L^{2}}$
at any metric $\psi$ is given by 
\[
d(\mathcal{F}_{k,L^{2}})_{|\psi}=\frac{1}{N_{k}}\rho_{k\psi}dV
\]
 where the function $\rho_{k\psi}$ is the restriction to the diagonal
of the point-wise norm of the Bergman kernel of the Hilbert space
$\left(H^{0}(X,kL),\left\Vert \cdot\right\Vert _{L^{2}(X,dV)}\right).$
The asymptotics of $\rho_{k\psi},$ when $\psi$ is a smooth metric
with strictly positive curvature, are well-known and in particular
give that 
\begin{equation}
(i)\,\lim_{k\rightarrow\infty}\frac{1}{N_{k}}\rho_{k\psi}dV=\frac{1}{V}(dd^{c}\psi)^{n}\,\,\,(ii)\,\frac{1}{N_{k}}\rho_{k\psi}\leq C\label{eq:B-T}
\end{equation}
in the weak topology. Using the defining property \ref{eq:def of energyfunc as primitive}
of the functional $\mathcal{E}$ (and integrating along a line segment
in the space of all positively curved metrics) this gives $\mathcal{F}_{k,L^{\infty}}[\psi]=\mathcal{E}(\psi)+o(1),$
which proves the first point in the theorem, thanks to \ref{eq:norm wrt phi as wrt P phi}
and \ref{eq:distortion of norms}. 

As for the differentiability in the second point it is proved in \cite{b-b}
in the general setting of $\theta-$psh functions, not necessarily
associated to a line bundle (i.e. for a a general big class $[\theta]\in H^{1,1}(X.\R)),$
using the orthogonality relation \ref{eq:og relation}. However in
the present line bundle setting an alternative proof using the general
Bergman kernel asymptotics in \cite{berm1} can be given. Indeed,
by \cite{berm1}, 
\begin{equation}
d(\mathcal{F}_{k,L^{2}})_{|\phi}=\frac{1}{V}(dd^{c}P\phi)^{n}+o(1)\label{eq:asympt of diff for L big}
\end{equation}
 for any smooth metric $\phi$ on $L$ (thus generalizing \ref{eq:B-T}).
Using \ref{eq:norm wrt phi as wrt P phi} and \ref{eq:distortion of norms}
again this implies the desired differentiability result. In the general
case of a line bundle $L$ which is merely big one cannot reduce the
problem to the case of smooth and positively curved metrics (since
such metrics will not, in general, exist). But the point is that,
as shown in \cite{berm1}, the asymptotics \ref{eq:asympt of diff for L big}
are always valid for any big line bundle, which is enough to conclude
(also using the general differentiability result in\cite{b-b}). 
\end{proof}
The next lemma provides the regularization property \ref{eq:density property}
in the present setting:
\begin{lem}
\label{lem:appr}Let $\mu$ be a probability measure such that $E_{\theta}(\mu).$
Then there exists a sequence $\mu_{j}$ of probability measures of
the form $\mu_{j}=MA(P_{\theta}u_{j}),$ for $u_{j}\in C^{0}(X),$
such that 
\[
\lim_{j\rightarrow\infty}\mu_{j}=\mu,\,\,\,\lim_{j\rightarrow\infty}E(\mu_{j})=\mu
\]
where the first convergence holds in the weak topology.\end{lem}
\begin{proof}
If $E(\mu)<\infty,$ then by Theorem \ref{thm:var sol of ma} we can
write $\mu=MA(\varphi)$ for a function $\varphi$ with finite energy.
Since the function $\varphi$ is usc it is a decreasing limit of continuous
functions $u_{j}$ on $X.$ It then follows, by monotonicity, that
the projections $Pu_{j}$ also decrease to $\varphi$ and hence, by
the continuity of mixed Monge-Ampère expression under monotone limits
\cite{begz} it follows that $\mu_{j}:=MA(Pu_{j})\rightarrow\mu$
and $E(\mu_{j})\rightarrow E(\mu),$ as desired. 
\end{proof}
The previous lemma can also be obtained form general properties of
Legendre transforms (see \cite[Lemma 3.1]{berm8}).

\section{\label{sec:The-LDP-for}The LDP for $\beta-$deformed determinantal
point processes }

Given a compact topological space $X$ we will denote by $\mathcal{M}(X)$
the space of all signed (Borel) measures on $X$ and by the $\mathcal{M}_{1}(X)$
the subspace of all probability measures, i.e. $\mu\geq0$ and $\int_{X}\mu=1.$
We endow $\mathcal{M}(X)$ with the weak topology, i.e. $\mu_{j}$
is said to converge to $\mu$ weakly in $\mathcal{M}(X)$ if 
\[
\left\langle u,\mu_{j}\right\rangle \rightarrow\left\langle u,\mu\right\rangle 
\]
 for any continuous function $u$ on $X,$ i.e. for any $u\in C^{0}(X),$
where $\left\langle u,\mu\right\rangle $ denotes the standard integration
pairing between $C^{0}(X)$ and $\mathcal{M}(X).$

\subsection{\label{sub:Probabilistic-preliminaries}Probabilistic preliminaries}

A \emph{probability space }is a space $\Omega$ equipped with a probability
measure $\mu.$ The space $\Omega$ is called the\emph{ sample space}
and a measurable subset $\mathcal{B}\subset\Omega$ is called an \emph{event}
with 
\[
\mbox{Prob}\mathcal{B}:=\mu(\mathcal{B}),
\]
interpreted as the probability of observing the event $\mathcal{B}$
when sampling from $(\mathcal{X},\Omega).$ A measurable function
$Y:\,\Omega\rightarrow\mathcal{Y}$ on a probability space $(\Omega,\mu)$
is called a\emph{ random element with values in $Y$ }and its\emph{
law} $\Gamma$ is the probability measure on $\mathcal{Y}$ defined
by the push-forward measure
\[
\Gamma:=Y_{*}\mu
\]
 (the law of $Y$ is often also called the distribution of $Y$).
A sequence of random elements $Y_{N}:\,\Omega_{N}\rightarrow\mathcal{Y}$
taking values in the same topological space $\mathcal{Y}$ are said
to\emph{ convergence in law towards a deterministic element $y$ in
$Y$ }if the corresponding laws $\Gamma_{N}$ on $\mathcal{Y}$ converge
to a Dirac mass at $y:$ 
\[
\lim_{N\rightarrow\infty}\Gamma_{N}=\delta_{y}
\]
 in the weak topology. If $\mathcal{Y}$ is a separable metric space
with metric $d$ then $Y_{N}$ converge in law towards the deterministic
element $y$ iff $Y_{N}$ \emph{converge in probability} towards $y,$
i.e. for any fixed $\epsilon>0$ i.e. 
\[
\lim_{N\rightarrow\infty}\mbox{Prob}\{d(Y_{N},y)>\epsilon\}=0.
\]

\subsubsection{Random point processes}

A\emph{ random point process} with $N$ particles on a space $X$
is, by definition, a probability measure $\mu^{(N)}$ on the $N-$fold
product $X^{N}$ (the $N-$particle space) which is symmetric, i.e.
invariant under action of the group $S_{N}$ by permutations of the
factors of $X^{N}.$ Its\emph{ j-point correlation measure }$\mu_{j}^{(N)}$
is the probability measure on $X^{j}$ defined as the push forward
of $\mu^{(N)}$ to $X$ under the map $X^{N}\rightarrow X^{j}$ given
by projection onto the first $j$ factors (or any $j$ factors, by
symmetry): 
\[
\mu_{j}^{(N)}:=\int_{X^{N-j}}\mu^{(N)}
\]
The\emph{ empirical measure} of a given random point process is the
following random measure 
\begin{equation}
\delta_{N}:\,\,X^{N}\rightarrow\mathcal{M}_{1}(X),\,\,\,(x_{1},\ldots,x_{N})\mapsto\delta_{N}(x_{1},\ldots,x_{N}):=\frac{1}{N}\sum_{i=1}^{N}\delta_{x_{i}}\label{eq:empirical measure text}
\end{equation}
on $(X^{N},\mu^{(N)}).$ The law of $\delta_{N}$ thus defines a probability
measure on the space$\mathcal{M}_{1}(X)$ that we shall denote by
$\Gamma_{N}.$ 
\begin{lem}
\label{lem:Conv in law equi to chaos}Consider a sequence of random
point processes with $N$ particles on $X.$ Then the corresponding
random measures$\delta_{N}$ converge in law towards a deterministic
measure $\mu$ iff 
\begin{equation}
\lim_{N\rightarrow\infty}\mu_{j}^{(N)}=\mu^{\otimes j},\label{eq:conv of j point correl}
\end{equation}
 weakly on $X^{j},$ for any fixed positive integer $j.$ \end{lem}
\begin{proof}
For completeness we give the simple proof of the convergence of the
$j-$point correlation measures under the assumption that $\delta_{N}$
converge in law towards a deterministic measure $\mu$ (which is the
direction we will be interested in). When $j=1$ the convergence \ref{eq:conv of j point correl}
means that, for any given $u\in C^{0}(X)$ the following holds:
\[
\lim_{N\rightarrow\infty}\int_{X^{N}}\frac{1}{N}\left(u(x_{1})+...+u(x_{N})\right)\mu^{(N)}=\int_{X}u\mu.
\]
Denoting by $\mathcal{U}$ the continuous function on $\mathcal{M}(X)$
defined by $\mathcal{U}(\mu):=\int_{X}u\mu,$ for a given $u\in C^{0}(X),$
the integral in the left hand side above may be written as 
\[
\int_{X^{N}}(\delta_{N}^{*}\mathcal{U})\mu^{(N)}=\int_{\mathcal{M}_{1}(X)}\mathcal{U}\Gamma_{N},\,\,\,\Gamma_{N}:=(\delta_{N*}\mu^{(N)}),
\]
 which, by assumption, converges to $U(\delta_{\mu}):=\int_{X}u\mu,$
as $N\rightarrow\infty,$ as desired. The case of $j>1$ is proved
in a similar way by replacing the linear function $\mathcal{U}$ with
a ``multinomial'' $\mathcal{U}_{j}(\mu):=\int_{X}u_{1}(x_{1})\mu\cdots\int_{X}u_{j}(x_{1})\mu$
determined by $j$ given elements $u_{j}$ in $C^{0}(X).$
\end{proof}

\subsubsection{The notion of a Large Deviation Principle (LDP)}

The notion of a \emph{Large Deviation Principle (LDP)}, introduced
by Varadhan, allows one to give a notion of exponential convergence,
which can be seen as an infinite dimensional version of the Laplace
principle \cite{d-z}. Let us first recall the general definition
of a Large Deviation Principle (LDP) for a general sequence of measures.
\begin{defn}
\label{def:large dev}Let $\mathcal{P}$ be a Polish space, i.e. a
complete separable metric space.

$(i)$ A function $I:\mathcal{\,P}\rightarrow]-\infty,\infty]$ is
a \emph{rate function} if it is lower semi-continuous. It is a \emph{good}
\emph{rate function} if it is also proper.

$(ii)$ A sequence $\Gamma_{k}$ of measures on $\mathcal{P}$ satisfies
a \emph{large deviation principle} with \emph{speed} $r_{k}$ and
\emph{rate function} $I$ if

\[
\limsup_{k\rightarrow\infty}\frac{1}{r_{k}}\log\Gamma_{k}(\mathcal{F})\leq-\inf_{\mu\in\mathcal{F}}I
\]
 for any closed subset $\mathcal{F}$ of $\mathcal{P}$ and 
\[
\liminf_{k\rightarrow\infty}\frac{1}{r_{k}}\log\Gamma_{k}(\mathcal{G})\geq-\inf_{\mu\in G}I(\mu)
\]
 for any open subset $\mathcal{G}$ of $\mathcal{P}.$ 
\end{defn}
The simplest instance of an LDP appears when $\mathcal{P}=\R^{n}$
and $\Gamma_{k}(y)$ is a probability measure of the form $\Gamma_{k}(y)=e^{-r_{k}I(y)}dy/Z_{k}$
for an appropriate lower semi-continuous function $I.$
\begin{example}
L\label{ex:Cramer}Let $Y_{1},...,Y_{N}$ be $N$ independent normal
standard random variables, i.e. Gaussian random variables with zero
mean and unit variance. Then the law of the corresponding sample mean
$\bar{Y}_{N}=(Y_{1}+...Y_{N})/N$ satisfies a LDP on $\R$ with rate
functional $I(y)=|y|^{2}/2$ and speed $N.$ Indeed, by definition
each $Y_{i}$ is the coordinate variable on the probability space
$(\R,\gamma),$ where $\gamma$ is the standard Gaussian probability
measure, i.e. $\gamma=(2\pi)^{-1/2}e^{-|y|^{2}}dy$ and an explicit
calculation gives the exact formula
\[
\Gamma_{N}:=(Y^{(N)})_{*}\gamma^{\otimes N}=(2\pi)^{-1/2}e^{-N|y|^{2}/2}dy,
\]
 which implies the LDP in question. More generally, by Cramér's theorem
\cite{d-z}, replacing $\gamma$ with any probability measure $\nu$
(with finite exponential moments) gives an LDP with speed $N$ and
a rate function $I(y)$ which, by inverting the Laplace transform
of $\nu,$ may be represented as follows, in terms of the one-dimensional
Legendre-Fenchel transform (see Remark \ref{rem:=00005BLegendre-transform-in}):
\[
I(y):=f^{*}(y),\,\,\,\,f(x):=\log\int_{\R}e^{\left\langle x,y\right\rangle }\nu(y),
\]
In particular, by convexity, $I(y)$ vanishes precisely on $df/dx_{|x=0}=\int y\nu.$
This implies the weak (and also strong) law of large numbers saying
that $\bar{Y}_{N}$ converges in probability (and even almost surely)
to the deterministic value $\int y\nu;$ indeed, the sample mean is
even exponentially concentrated around its expectation. 
\end{example}
Let us also mention the following classical infinite dimensional generalization
of the previous LDP due to Mogulskii (see \cite[Chapter 5]{d-z}):
\begin{example}
Set $\bar{Y}_{N}(t):=\bar{Y}_{[tN]}$ for $t\in[0,1]$ where $[c]$
denotes the integer part of $c.$ If $Y_{i}$ are standard independent
normal random variables, then the law $\Gamma_{N}$ of $\bar{Y}_{N}(t),$
viewed as random element with values in the space $C^{0}[0,1]_{0}$
of all continuous functions on $[0,1]$ such that $u(0)=0,$ satisfies
an LDP with speed $N$ and rate functional defined by 
\[
I(u):=\int_{0}^{1}|\frac{du}{dt}|^{2}dt,
\]
 if the function $u\in C^{0}[0,1]_{0}$ has a distributional derivative
in $L^{2}[0,1]$ and otherwise $I(u)=\infty.$ Interpreting the parameter
$t$ as time the random function $\bar{Y}_{N}(t)$ can be viewed as
a sample path for a random walk on $\R$ starting at the origin.
\end{example}
We will be mainly interested in the case when $\Gamma_{k}$ is a probability
measure (which implies that $I\geq0$ with infimum equal to $0).$
Then it will be convenient to use the following alternative formulation
of a LDP (see Theorems 4.1.11 and 4.1.18 in \cite{d-z}):
\begin{prop}
\label{prop:d-z}Let $\mathcal{P}$ be a compact metric space and
denote by $B_{\epsilon}(\nu)$ the ball of radius $\epsilon$ centered
at $\nu\in\mathcal{P}.$ Then a sequence $\Gamma_{N}$ of probability
measures on $\mathcal{P}$ satisfies a LDP with speed $r_{N}$ and
a rate functional $I$ iff 
\begin{equation}
\lim_{\epsilon\rightarrow0}\liminf_{N\rightarrow\infty}\frac{1}{r_{N}}\log\Gamma_{N}(B_{\epsilon}(\nu))=-I(\nu)=\lim_{\epsilon\rightarrow0}\limsup_{N\rightarrow\infty}\frac{1}{r_{N}}\log\Gamma_{N}(B_{\epsilon}(\nu))\label{eq:ldp in terms of balls in prop}
\end{equation}
In particular, if the rate functional $I$ has a unique minimizer
$\mu_{min},$ then $\Gamma_{N}\rightarrow\delta_{\mu_{min}},$ weakly,
as $N\rightarrow\infty.$
\end{prop}
In the present setting $\Gamma_{N}$ will arise as the law of the
empirical measures $\delta_{N}$ and the rate functional $I$ will
be shown to have a unique minimizer, which will thus imply that $\delta_{N}$
converges in law towards the deterministic measure $\mu_{min}.$

We will have great use for the following classical result of Sanov,
which is the standard example of an LDP for point processes.
\begin{prop}
(Sanov) \label{prop:sanov}Let $X$ be a topological space and $\mu_{0}$
a finite measure on $X.$ Then the laws $\Gamma_{N}$ of the empirical
measures $\delta_{N}$ defined with respect to the product measure
$\mu_{0}^{\otimes N}$ on $X^{N}$ satisfy an LDP with speed $N$
and rate functional the relative entropy $D_{\mu_{0}}.$\end{prop}
\begin{proof}
We recall that, as shown in \cite{d-z}, the proof can be obtained
from the infinite dimensional generalization of Cramér's theorem in
example \ref{ex:Cramer}. Indeed, assuming for simplicity that $X$
is compact (which will be the case in the present setting) the topological
dual of the topological vector space $\mathcal{M}(X),$ consisting
of signed Borel measures, may be identified with $C^{0}(X)$ and the
rate functional $I$ on $\mathcal{M}(X)$ may be written as the Legendre-Fenchel
transform $f^{*}(\mu),$ where $f(u)=\log\int e^{\left\langle u,\delta_{x}\right\rangle }\mu_{0}=\log\int e^{u(x)}\mu_{0}.$
A direct computation then reveals that $f^{*}$ is the relative entropy. 
\end{proof}
We recall that the \emph{relative entropy} $D_{\mu_{0}}$ (also called
the \emph{Kullback\textendash Leibler divergence }or the\emph{ information
divergence} in probability and information theory) is the functional
on $\mathcal{M}_{1}(X)$ defined by 
\begin{equation}
D_{\mu_{0}}(\mu):=\int_{X}\log\frac{\mu}{\mu_{0}}\mu,\label{eq:def of rel entropy}
\end{equation}
 when $\mu$ has a density $\frac{\mu}{\mu_{0}}$ with respect to
$\mu_{0}$ and otherwise $D_{\mu_{0}}(\mu):=\infty.$ When $\mu_{0}$
is a probability measure,$D_{\mu_{0}}(\mu)\geq0$ and $D_{\mu_{0}}(\mu)=0$
iff $\mu=\mu_{0}$ (by Jensen's inequality).

\subsubsection{Gibbs measures}

Let $X$ be a compact topological space endowed with a measure $\mu_{0}$
and $H^{(N)}(x_{1},...,x_{N})$ a symmetric function on $X^{N}.$
As recalled in Section \ref{sub:Relations-to-statistical}, for any
given positive number $\beta$ the corresponding \emph{Gibbs measure
at inverse temperature $\beta$} 
\[
\mu_{\beta}^{(N)}:=e^{-\beta H^{(N)}}\mu_{0}^{\otimes N}/Z_{N,\beta},\,\,\,Z_{N.\beta}:=\int_{X^{N}}e^{-\beta_{N}H^{(N)}}\mu_{0}^{\otimes N},
\]
defines a random point process on $X$ with $N$ particles, assuming
that the normalizing constant $Z_{N,\beta}$ (called the\emph{ partition
function}) is finite. We note the following lemma which allows one
to extend Prop \ref{prop:d-z} to the non-normalized measures $(\delta_{N})_{*}e^{-\beta H^{(N)}}\mu_{0}^{\otimes N}$
(see \cite{berm8} for the simple proof). 
\begin{lem}
\label{lem:ldp for nonnormalized meas}Assume that $\left|\log Z_{N,\beta}\right|\leq CN.$
Then the measures $(\delta_{N})_{*}e^{-\beta H^{(N)}}\mu_{0}^{\otimes N}$
satisfy the asymptotics \ref{eq:ldp in terms of balls in prop} for
any $\nu\in\mathcal{M}_{1}(X)$ with rate functional $\tilde{I}(\mu)$
and speed $N$ iff the probability measures $(\delta_{N})_{*}\mu_{\beta}^{(N)}$
on $\mathcal{M}_{1}(X)$ satisfy an LDP at speed $N$ with rate functional
$I:=\tilde{I}-C_{\beta},$ where $C_{\beta}:=\inf_{\mathcal{\mu\in}\mathcal{M}(X)}I(\mu).$
\end{lem}

\subsection{Large deviations for $\beta-$deformation of Vandermonde type determinants
attached to a big line bundle $L$ }

As above we let $L\rightarrow X$ be a given big line bundle over
a compact complex manifold, $\left\Vert \cdot\right\Vert $ is a smooth
Hermitian metric on $L$ (whose curvature current will be denoted
by $\theta)$. We also fix a a finite measure $\mu_{0}$ on $X$ a
positive number $\beta.$ To this data we may associate the following
sequence of probability measures on $X^{N_{k}}:$ 
\begin{equation}
\mu^{(N_{k},\beta)}:=\frac{\left\Vert (\det S^{(k)})(x_{1},x_{2},...x_{N_{k}})\right\Vert ^{2\beta/k}\mu_{0}^{\otimes N_{k}}}{Z_{N_{k},\beta}}\label{eq:prob measure general intro-text}
\end{equation}
where we recall that $N_{k}$ is the dimension of $H^{0}(X,kL)$ and
$\det S^{(k)}$ is a generator of the corresponding determinant line
$\Lambda^{N_{k}}H^{0}(X,kL)$ viewed as a one-dimensional subspace
of $H^{0}(X^{N_{k}},(kL)^{\boxtimes N_{k}})$ (the totally anti-symmetric
part). As usual, $Z_{N_{k},\beta}$ is the normalizing constant (partition
function): 
\begin{equation}
Z_{N_{k},\beta}:=\int_{X^{N_{k}}}\left\Vert \det S^{(k)}\right\Vert ^{2\beta/k}\mu_{0}^{\otimes N_{k}}\label{eq:def of partion function in beta-setting text}
\end{equation}
By homogeneity $\mu^{(N_{k},\beta)}$ is independent of the choice
of generator $\det S^{(k)}.$ It will be convenient to take $\det S^{(k)}$
to be the generator determined by a basis in $H^{0}(X,kL)$ which
is orthonormal with respect to the $L^{2}-$product determined by
$(\left\Vert \cdot\right\Vert ,dV)$ for any fixed volume form $dV$
on $X.$ 
\begin{rem}
\label{rem:More-generally,-one}More generally, one can let $\beta$
depend on $N$ (i.e. on $k)$ in the definition of $\mu^{(N_{k},\beta)}.$
When $\beta_{N_{k}}=k$ the corresponding probability measure $\mu^{(N_{k},\beta_{k})}$
defines a determinantal point processes, i.e. its density can be written
as 
\[
\left\Vert \det_{i,j\leq N}(K^{(k)}(x_{i},x_{j}))\right\Vert /N_{k}!,
\]
 where $K^{(k)}(x,y)$ denotes the kernel of the orthogonal projection
onto the space $H^{0}(X,kL)$ viewed as a subspace of the space $C^{\infty}(X,kL)$
of all smooth sections equipped with the $L^{2}-$norm determined
by $(\left\Vert \cdot\right\Vert ,\mu_{0}).$ There is an extensive
literature concerning general properties of determinantal point process
( defined with respect to general Hilbert spaces of functions in $L^{2}(\mu_{0})$
for a given measure $\mu_{0}$ on a space $X)$ - for example, all
the $j-$point correlation measures can be expressed as determinants
involving the kernel $K(x,y).$ The LDP for determinantal point processes
associated to a line bundle $L\rightarrow X$ as above was established
in \cite{berm 1 komma 5} for very general measures $\mu_{0}$ and
can be viewed as a zero-temperature limit since $\beta_{N_{k}}\rightarrow\infty$
(compare Section \ref{sub:Relations-to-statistical}). However, the
present setting, which more generally applies when $\beta_{k}$ has
the asymptotics $\beta_{k}=\beta+o(1)$ for a positive number $\beta,$
appears to be substantially more involved from an analytic point of
view and we will only be able to establish the LDP in question for
sufficiently regular measures $\mu_{0}.$ The case $\beta=1$ is singled
out by the fact that it allows the construction of canonical point
processes independent of any geometric back-ground data, when $L$
is the canonical line bundle on a variety of positive Kodaira dimension
(as explained in Section \ref{sub:Canonical-point-processes klt}).
\end{rem}
The probability measure above is thus the Gibbs measure defined by
the following Hamiltonian on $X^{N_{k}}:$ 
\begin{equation}
E^{(N_{k})}(x_{1},x_{2},...x_{N_{k}}):=-\frac{1}{k}\log\left\Vert (\det S^{(k)})(x_{1},x_{2},...x_{N_{k}})\right\Vert ^{2}\label{eq:def of Hamtiltonian in complex setting text}
\end{equation}
Since the metric $\left\Vert \cdot\right\Vert $ is determined up
to a multiplicative constant by its curvature current $\theta,$ the
probability measures above thus only depends on the metric through
its curvature and is hence determined by the triple $(\mu_{0},\theta,\beta).$
To the latter triple we may also attach the following free energy
type functional on the space $\mathcal{M}_{1}(X)$ of all probability
measures on $X:$ 
\begin{equation}
F_{\beta}=E_{\theta}+\frac{1}{\beta}D_{\mu_{0}}\label{eq:free energy complex setting text}
\end{equation}
 Note that the energy functional $E_{\theta}$ is minimized on the
measure $MA(v_{\theta})$ (see Prop \ref{prop:pluri energy as legendre}),
while the entropy functional $D_{\mu_{0}}$ is minimized on $\mu_{0}.$
By \cite{berm6,bbgez} the minimizers of $F_{\beta}$ may be identified
with solutions to a complex Monge-Ampère equation (the existence of
smooth solutions for an ample line bundle was first shown in the seminal
works of Aubin \cite{au} and Yau \cite{y}):
\begin{prop}
\label{prop:min of free energy}Let $[\theta]$ be a big class (for
example, $[\theta]=c_{1}(L)$ for $L$ a big line bundle) and consider
the free energy functional $F_{\beta}$ attached to the triple $(\mu_{0},\theta,\beta)$
(for $\beta>0),$ where $\mu_{0}$ has finite energy. Then any minimizer
$\mu_{\beta}$ of $F_{\beta}$ on $\mathcal{M}_{1}(X)$ can be written
as $\mu_{\beta}=MA(u_{\beta})$ where $u_{\beta}$ is the unique finite
energy solution of the equation 
\begin{equation}
MA(u_{\beta})=e^{\beta u_{\beta}}\mu_{0}\label{eq:ma eq in text}
\end{equation}
\end{prop}
\begin{proof}
The case of a Kähler class was proved in \cite{berm6} and the proof
generalizes word for word to the case of a big case. For completeness
we outline the proof. First of all, by the strict convexity of the
relative entropy and the convexity of $E$ the functional $F_{\beta}$
is strictly convex and in particular if a minimizer $\mu_{min}$ exists,
then it it unique. In fact, by the lower semi-continuity of $F_{\beta}$
there always exists a minimizer and as we will next show it can be
obtained from the equation in the proposition. To see this we recall
that, as shown in \cite{berm6}, given any $\mu\in\mathcal{P}(X)$
such that $E(\mu)<\infty,$ the function$-\varphi_{\mu}$ is a subgradient
for the convex functional $E(\mu)$ at $\mu$ in the sense that 
\begin{equation}
E(\nu)\geq E(\mu)+\int_{X}(-\varphi_{\mu})(\nu-\mu),\label{eq:conv ineq for E}
\end{equation}
for any $\nu\in\mathcal{P}(X)$ (the proof in the case of a big class
is the same, as it relies on \cite{bbgz} which holds in the general
setting of a big class). Moreover, it is a classical fact that if
$D_{\mu_{0}}(\mu)<\infty$ then $\log(\mu/\mu_{0})$ is a subgradient
(and even a gradient) for $D_{\mu_{0}}(\mu).$ Hence, if $\mu$ satisfies
the equation 
\begin{equation}
\frac{1}{\beta}\log(\mu/\mu_{0})-\varphi_{\mu}=0,\label{eq:ma eq in proof in terms of mu}
\end{equation}
 then $F_{\beta}(\nu)\geq F_{\beta}(\mu)$ for any $\nu\in\mathcal{P}(X),$
i.e. $\mu$ minimizes $F_{\beta}(\mu).$ Now, if $\varphi\in\mathcal{E}^{1}(X,\theta)$
solves the equation \emph{\ref{eq:ma eq in text}} then $\mu:=MA(\varphi_{\mu})$
solves the equation \ref{eq:ma eq in proof in terms of mu}, so all
that remains is to find a solution to the equation \emph{\ref{eq:ma eq in text}}
in $\mathcal{E}^{1}(X,\theta).$ But the existence of such a solution
follows from the results in \cite{bbgz} (generalizing Theorem \ref{thm:var sol of ma}
which corresponds to $\beta=0).$ In a nutshell, the function $\varphi$
is obtained by maximizing the following lsc functional on $\mathcal{E}^{1}(X,\theta):$
\[
\mathcal{E}(\psi)-\frac{1}{\beta}\log\int e^{\beta\psi}\mu_{0}.
\]

\end{proof}
For the complex geometric applications it will be adequate to consider
measures of the following form: 
\begin{equation}
\mu_{0}=e^{\psi_{+}-\psi_{-}}dV,\label{eq:mu zero in terms of quasipsh}
\end{equation}
 where $dV$ is a volume form on $X$ and $\psi_{\pm}$ are quasi-psh
functions. 
\begin{rem}
When $L$ is ample and $\mu_{0}$ is a singular volume, which is singular
and degenerate along a klt divisor $\Delta,$ in $X$ it is well-known
that the solution is smooth away from the support of $\Delta$ (more
generally, the regularity holds on the intersection with the ample
locus of $L$ when $L$ is assumed merely nef \cite{begz}). However,
for the purposes of the present paper such regularity properties will
not play any role.\end{rem}
\begin{thm}
\label{thm:big line bundle text}Let $L\rightarrow X$ be a big line
equipped with a continuous Hermitian metric $\left\Vert \cdot\right\Vert $
and $\mu_{0}$ a finite measure on $X$ of the form \ref{eq:mu zero in terms of quasipsh}.
Fix a positive number $\beta.$ Then the empirical measures of the
corresponding random point processes on $X$ converge in probability
towards the measure $\mu_{\beta},$ where $\mu_{\beta}=MA(u_{\beta})$
for the unique solution $u_{\beta}$ of the complex Monge-Ampère equation
\ref{eq:ma eq in text}. More precisely:
\begin{itemize}
\item The non-normalized measures $(\delta_{N})_{*}\left(\left\Vert \det S^{(k)}\right\Vert ^{2\beta/k}\mu_{0}^{\otimes N_{k}}\right),$
where \emph{$\det S^{(k)}$ an element in the determinant line of
$H^{0}(X,kL)$ induced by a basis in $H^{0}(X,kL)$ which is orthonormal
with respect to the $L^{2}-$norm determined by }$(\left\Vert \cdot\right\Vert ,\mu_{0})$
satisfy a LDP on $\mathcal{M}_{1}(X)$ with rate functional $F_{\beta}$
and the corresponding partition functions satisfy 
\[
-\lim_{N_{k}\rightarrow\infty}\frac{1}{N_{k}}\log Z_{N_{k},\beta}=\inf_{\mathcal{M}_{1}(X)}F_{\beta}=:C_{\beta}
\]

\item As a consequence, the laws of the empirical measures of the random
point processes defined by formula \ref{eq:prob measure general intro-text}
satisfy a large deviation principle (LDP) with speed $\beta N$ and
rate functional $F_{\beta}(\mu)-C_{\beta}.$ 
\end{itemize}
\end{thm}
\begin{proof}
When $\mu_{0}$ is a volume form the proof of the first point was
outlined in Section \ref{sec:Outline-of-the} and the full proof appears
in the companion paper \cite{berm8}. Here we just point out that
the proof also applies to more general $\mu_{0}$ of the form \ref{eq:mu zero in terms of quasipsh}.
First of all, the proof of the upper bound was given in Section \ref{sec:Outline-of-the}
and did not use any properties of the measure $\mu_{0}$ at all. As
for the lower bound we first observe that, upon replacing $\beta H^{(N)}$
with $\beta H^{(N)}(x_{1},...x_{N})-\psi_{+}(x_{1})-....-\psi_{+}(x_{N}),$
we may as well assume that $\mu_{0}=e^{-\psi_{-}}dV.$ Now, since
$\psi_{-}$ is assumed usc there is a sequence $v_{j}\in C^{0}(X)$
decreasing to $\psi_{-}$ and we set $\mu_{0,j}:=e^{-v_{j}}dV.$ We
thus have $e^{-\beta H^{(N)}}\mu_{0}^{\otimes N}\geq e^{-\beta H^{(N)}}\mu_{0,j}^{\otimes N},$
so that we can apply the LDP with respect to $\mu_{0,j}$ for a fixed
$j$ to get that the liminf in the LDP with respect to $\mu_{0}$
is bounded from below by $\mathcal{F}^{*}(\mu)+D_{\mu_{0,j}}(\mu)/\beta.$
Finally, the proof is concluded by letting $j\rightarrow\infty$ and
using the monotone convergence theorem of integration theory.The LDP
in the second point then follows from Lemma \ref{lem:ldp for nonnormalized meas}.\end{proof}
\begin{rem}
\label{Remark 26. transf. of free en}Note that when $\beta=1$ the
probability measure $\mu^{(N,\beta)}$on $X^{N}$ is invariant under
the transformation $(\left\Vert \cdot\right\Vert ^{2},\mu_{0})\mapsto((\left\Vert \cdot\right\Vert ^{2}e^{-v},e^{v}\mu_{0})$
of the defining data. Hence, by the previous theorem the corresponding
free energy functional $F_{\beta}$ is invariant under $(\theta,\mu_{0})\mapsto(\theta+ddv,e^{v}\mu_{0}),$
up to an additive constant. In fact, this is easy to see directly,
since it follows from the definitions that $D_{e^{v}\mu_{0}}(\mu)=D_{\mu_{0}}(\mu)-\int v\mu$
and $E_{\theta+dd^{c}v}(\mu)=E_{\theta}(\mu)+\int v\mu-C_{v}$ (where
the constant $C_{v}$ ensures that the infimum of $E_{\theta+dd^{c}v}$
vanishes.
\end{rem}

\subsection{\label{sub:Identification-of-the}Identification of $F_{\beta}$
with Mabuchi's K-energy functional}

Now assume that $L$ is ample and fix data $(\left\Vert \cdot\right\Vert ,dV,\beta)$
where $dV$ denotes a fixed volume form. Since $L$ is ample we may
assume that the curvature form of $\left\Vert \cdot\right\Vert $
is a Kähler form that we denote by $\omega_{0}.$ The probability
measures \ref{eq:prob measure general intro-text} only depends on
the data through the following two form 

\begin{equation}
\eta:=\beta\omega_{0}+\ensuremath{\mbox{Ric}\,}\ensuremath{\mu_{0},}\label{eq:def of eta}
\end{equation}
where $\ensuremath{\mbox{Ric\,}}\ensuremath{\mu_{0}}$denotes the
curvature form of the metric on $-K_{X}$ induced by the volume form
$\mu_{0}$ (compare Remark \ref{Remark 26. transf. of free en}).
It is a basic fact that the Monge-Ampère equation \ref{eq:ma eq in text}
is equivalent to the\emph{ twisted Kähler-Einstein equation} 
\begin{equation}
\mbox{\ensuremath{\mbox{Ric}\,}\ensuremath{\omega}}=-\beta\omega+\eta,\label{eq:tw ke eq in text}
\end{equation}
 (see \cite[Lemma 5.1]{berm8}).

We denote by $\mathcal{H}(X,\omega_{0})$ the interior of $PSH(X,\omega_{0})\cap C^{\infty}(X).$
The map $\varphi\mapsto\omega_{\varphi}$ identifies $\mathcal{H}(X,\omega_{0})/\R$
with the space of Kähler metrics in $c_{1}(L).$ By the Calabi-Yau
theorem \cite{y} any normalized volume form $\mu$ on $X$ is the
volume form of a unique Kähler form $\omega\in c_{1}(L),$ i.e. $\mu=\omega^{n}$
for a Kähler form $\omega\in c_{1}(L).$ As explained in \cite{berm6},
under the corresponding Calabi-Yau isomorphism $\mu\mapsto\omega,$
the functional $F_{\beta}$ may be identified with the twisted version
of Mabuchi's K-energy $\kappa$ of $\omega,$ i.e. 
\begin{equation}
F_{\beta}(\omega^{n})=\mbox{\ensuremath{\kappa}(\ensuremath{\omega}), }\label{eq:F is K-energ}
\end{equation}
We recall that the functional $\kappa$ was originally defined in
\cite{m1} for a general polarized manifold $(X,L)$ (when $\eta=0)$
by the implicit property 
\[
d(\kappa(\omega_{\varphi})_{|\varphi}=-(R_{\omega_{\varphi}}-C)\omega_{\varphi}^{n},
\]
 for $\varphi\in\mathcal{H}(X,\omega_{0}),$ where $R_{\omega_{\varphi}}$
is the (normalized) scalar curvature of the Kähler metric $\omega_{\varphi}$
and $C$ is a constant only depending on the cohomology of $(X,L).$
This formula thus defines $\kappa$ up to a normalizing constant.
When $L=K_{X}$ the formula \ref{eq:F is K-energ} then follows directly
from the Chen-Tian formula for $\kappa.$ For the convenience of the
reader we note that a direct proof of formula \ref{eq:F is K-energ}
can be given using Legendre transforms as follows (see Proposition
\ref{prop:pluri energy as legendre}). When $\varphi\in\mathcal{H}(X,\omega_{0})$
setting $\mu=\omega_{\varphi}^{n}$ gives 
\[
E_{\omega_{0}}(\mu)=f^{*}(df_{|(-\varphi)}),
\]
 and hence, by basic properties of Legendre transform, $d(E_{\omega_{0}}(\mu)_{|\mu}=-\varphi.$
Moreover, since clearly $d(D_{dV}(\mu))_{|\mu}=\log(\mu/dV)$ differentiating
the map $\varphi\mapsto\omega_{\varphi}^{n}$ from $\mathcal{H}(X,\omega_{0})$
into $\mathcal{M}_{1}(X)$ gives, using the chain rule, $d(F_{\beta}(\omega_{\varphi}^{n}))_{|\varphi}/(n-1)=$
\[
=\omega_{\varphi}^{n-1}\wedge\frac{i}{\pi}\partial\bar{\partial}(-\varphi+\beta^{-1}\log(\omega_{\varphi}^{n}/dV)=\omega_{\varphi}^{n-1}\wedge(-\omega_{\varphi}-\alpha-\beta^{-1}\mbox{Ric \ensuremath{\omega_{\varphi}})}=
\]
 which proves the twisted generalization of formula \ref{eq:F is K-energ}
(with $C=1)$.
\begin{rem}
In the case when $L$ is ample an alternative proof of the fact that
$\omega_{\beta}^{n}$ minimizes $F_{\beta}$ on $\mathcal{M}_{1}(X)$
can be given by using that $\omega_{\beta}$ is a critical point of
$\kappa$ and hence, by convexity, minimizes $\kappa$ on $\mathcal{H}(X,\omega_{0}).$
Accordingly, the Calabi-Yau isomorphism $\omega\mapsto\omega^{n}$
shows that $\omega_{\beta}^{n}$ minimizes the restriction of $F_{\beta}$
to the subspace of all volume forms in $\mathcal{M}_{1}(X).$ However,
showing that the infimum of $F_{\beta}$ over all of $\mathcal{M}_{1}(X)$
coincides with the infimum over the subspace of volume forms requires
the following non-trivial fact: any $\mu$ such that $E(\mu)<\infty$
can be written as a weak limit of volume forms $\mu_{j}$ such that
$E(\mu_{j})\rightarrow E(\mu)$ and $D_{dV}(\mu_{j})\rightarrow D_{dV}(\mu)$
(see \cite{bdl} where more general results are obtained).
\end{rem}

\section{\label{sec:Canonical-random-point}Canonical random point processes
on varieties of positive Kodaira dimension and log pairs}

\subsection{\label{sub:Canonical-point-processes klt}The case of a klt pair
$(X,D)$ of log general type}

Assume given a smooth log pair $(X,D),$ where $D$ is klt $\Q-$divisor
and assume that $(X,D)$ is of log general type, i.e. $L:=K_{X}+D$
is big. To the pair $(X,D)$ we may attach a canonical sequence of
random point processes on $X$ as follows. Fix a smooth metric on
$L$ represented by a weight $\phi_{0}.$ It determines a singular
volume form $\mu_{(\phi_{0},D)}$ locally represented as 
\[
\mu_{(\phi_{0},D)}=e^{\phi_{0}-\phi_{D}}idz\wedge d\bar{z},
\]
 where $idz\wedge d\bar{z}$ is a short hand for the local Euclidean
volume form determines by the local holomorphic coordinates $z$ and
$\phi_{0}$ is the corresponding local representation of the weight.
Then it if follows immediately from the definitions that the probability
measure in formula \ref{eq:prob measure general intro-text} determined
by the triple $(\mu,\phi_{0},\beta)=(\mu_{(\phi_{0},D)},\phi_{0},1)$
is independent of $\phi_{0}$ and thus canonically attached to $(X,D).$
In fact, it coincides with the probability measure of the canonical
random point defined in the introduction of the paper. The point is
that if $s_{k}$ is a holomorphic section of $k(K_{X}+D)$ then the
measure 
\[
\left\Vert s_{k}\right\Vert _{k\phi_{0}}^{2/k}\mu_{(\phi_{0},D)}=\left\Vert s_{k}\right\Vert _{k\phi_{0}}^{2/k}\mu_{(\phi_{0},D)}=\left|s_{k}\right|^{2/k}e^{-\phi_{D}}idz\wedge d\bar{z},
\]
 is clearly independent of $\phi_{0}.$

If $u_{KE}$ is the solution of the corresponding Monge-Ampère equation
\ref{eq:ma eq in text} (with $\theta=dd^{c}\phi_{0}$ and $\mu_{0}=\mu_{(\phi_{0},D)})$
then it will sometimes be convenient to rewrite the equation in terms
of the corresponding weight $\phi_{KE}:=\phi_{0}+u_{KE}:$ 
\begin{equation}
(dd^{c}\phi_{KE})^{n}=e^{\phi_{KE}-\phi_{D}}idz\wedge d\bar{z}\label{eq:k-e equation in terms of weights}
\end{equation}
Its curvature current $\omega_{KE}:=dd^{c}\phi_{KE}(=\theta+dd^{c}u_{KE})$
satisfies the following log Kähler-Einstein equation associated to
$(X,D):$ 
\begin{equation}
\mbox{Ric}\,\ensuremath{\omega_{KE}=-\omega_{KE}+[D]}\label{eq:k-e eq for current and pair text}
\end{equation}
where $\mbox{Ric}\ensuremath{\omega}$ denotes the Ricci curvature
of $\omega$ viewed as a current on $X.$ See \cite{ber-gu} for the
precise meaning of the previous equation in the general setting when
$K_{X}+D$ is merely assumed big. Anyway, for $K_{X}+D$ semi-ample
and big (or nef and big) it was shown in \cite{bbgez} that the solution
$\omega$ is smooth on the log regular locus (i.e. on $X-D)$ and
defines a bona fide Kähler-Einstein metric there and its potential
$u_{KE}$ is globally continuous on $X.$ Moreover, in the case when
$K_{X}+D$ is ample the current $\omega$ globally defines a singular
Kähler-Einstein metric with edge-cone singularities along $D$ (see
\cite{jmr,cgh}). Anyway, in the present setting will not need any
regularity properties of $\omega_{KE}.$ 

The free energy functional

\begin{equation}
F=E_{dd^{c}\phi_{0}}+D_{\mu_{(\phi_{0},D)}}.\label{eq:free energy in klt case}
\end{equation}
determined by the back-ground data $(\mu,\phi_{0},\beta)=(\mu_{(\phi_{0},D)},\phi_{0},1)$
is independent of $\phi_{0}$ modulo an additive constant (as follows
from the transformation property pointed out in Remark \ref{Remark 26. transf. of free en}
). In fact, when $K_{X}+D$ is ample the functional $F$ corresponds
to the log version $\kappa_{(X.D)}$ of Mabuchi's K-energy functional
in the sense that 
\[
F(\frac{\omega^{n}}{V})=\kappa_{(X.D)}(\omega)
\]
(we recall that $\kappa_{(X.D)}$ is also only defined up to an additive
constant). This can be seen as a generalization of a formula of Tian
and Chen for the K-energy (see \cite{berm6,bbgez} and references
therein).
\begin{thm}
\label{thm:conv for log general type}Let $(X,D)$ be a smooth klt
pair of general type. Then the empirical measures of the corresponding
canonical random point processes on $X$ converge in probability towards
the normalized volume form $dV_{KE}$ of the Kähler-Einstein metric
on $(X,D).$ More precisely, the laws of the empirical measures satisfy
a large deviation principle with speed $N_{k}$ and rate functional
$I(\mu):=F(\mu)-C,$ where $C=F(dV_{KE})$ and $\kappa_{(X.D)}(\omega):=I(\omega^{n}/V)$
coincides with Mabuchi's (log) K-energy of $\omega$ normalized so
that $\kappa_{(X.D)}(\omega_{KE})=0.$\end{thm}
\begin{proof}
Setting $(\mu,\phi_{0},\beta)=(\mu_{(\phi_{0},D)},\phi_{0},1)$ this
is a direct consequence of Theorem \ref{thm:big line bundle text}.
\end{proof}
As a rather direct consequence of the previous theorem we get the
following
\begin{cor}
\label{cor:conv of canoc seq of currents klt}Let $(X,D)$ be a smooth
klt pair of general type. Then the first correlation measures $\nu_{k}:=\int_{X^{N-1}}\mu^{(N_{k})}$
of the canonical point processes define a sequence of canonical measures
on $X$ converging weakly to $dV_{KE}$. Moreover, the curvature forms
of the corresponding metrics on $K_{X}$ defined by the sequence $\nu_{k}$
converge weakly to the unique Kähler-Einstein metric $\omega_{KE}$
on $X.$\end{cor}
\begin{proof}
First observe that, by definition, the one point-correlation measure
we may be written as
\[
\mu_{1}^{(N_{k})}=e^{\phi_{k}}idz\wedge d\bar{z}=e^{u_{k}}dV
\]
 where $\phi_{0}$ is a fixed smooth weight on $K_{X},$ $dV=\mu_{\phi_{0}}$
and $u_{k}\in PSH(X,\theta),$ for $\theta=dd^{c}\phi_{0}.$ In particular,
$\int_{X}e^{u_{k}}dV=1$ and hence by Jensen's inequality, $\sup_{X}u_{k}\leq C_{0}.$
But, by standard compactness results for $\theta-$psh functions \cite{g-z},
it follows that either $u_{k}$ converges in $L^{1}(X)$ towards some
$u\in PSH(X,\theta),$ or there is a subsequence $u_{k_{j}}$ such
that $u_{k_{j}}\rightarrow-\infty$ uniformly. But the latter alternative
is not compatible with the condition $\int_{X}e^{u_{k}}dV=1$ and
hence $u_{k}\rightarrow u$ in $L^{1}(X).$ Equivalently, this means
that $\phi_{k}\rightarrow\phi$ in $L_{loc}^{1}$ where $\phi$ is
a weight on $K_{X}$ with positive curvature current. On the other
hand, by the previous theorem $\mu_{1}^{(N_{k})}\rightarrow e^{\phi_{KE}}idz\wedge d\bar{z,}$
where $dd^{c}\phi_{KE}=\omega_{KE}$ (compare Lemma \ref{lem:Conv in law equi to chaos}).
But, since a subsequence of $\phi_{k}$ converges a.e. on $X$ it
then follows that $\phi_{KE}=\phi.$ In particular, $dd^{c}\phi_{k}\rightarrow dd^{c}\phi_{KE}=\omega_{KE}$
weakly and that concludes the proof.
\end{proof}

\subsection{\label{sec:Varieties-of-positive}Varieties of positive Kodaira dimension}

\subsubsection{Birational setup}

Let us start by recalling the standard setup in birational geometry.
Let $X$ and $X'$ be (normal) projective varieties. A \emph{rationa}l
mapping $F$ from $X$ to $X',$ denoted by a dashed arrow $X\dashrightarrow X',$
is defined by a morphism $F:U\rightarrow X'$ from a Zariski open
subset $U$ of $X.$ It is called \emph{birational} if it has an inverse.
Then there is a maximal Zariski open subset $U\subset X,$ where $F$
defines a well-defined isomorphism onto its image (the complement
of $U$ is called the exceptional locus of $F).$ Given a\emph{ }rationa\emph{l}
mapping $F$ from $X$ to $X'$ and a probability measure $X$ which
is is absolutely continuous with respect to Lebesgue measure, we can
define $F_{*}\mu$ by pushing forward the restriction of $\mu$ to
any Zariski open subset $U$ where $F$ is well-defined and pull-backs
of such measures can be similarly defined. If $F:\,X\dashrightarrow X'$
is birational then there exists a non-singular variety $Z$ and birational
morphisms $f:\,Z\rightarrow X$ and $f':\,Z\rightarrow X'$ such that
$f^{'}=F\circ f$ (in fact, $f$ and $f'$ can even be obtained as
a sequence of blow-ups and blow-downs respectively). The \emph{Kodaira
dimension} $\kappa(X)$ of an $n-$dimensional (say non-singular)
variety $X$ is the birational invariant defined as the smallest number
$\kappa\in\{-\infty,0,1.,...,n\}$ such that $N_{k}=O(k^{\kappa}),$
where $N_{k}$ denotes the $k$ th plurigenus of $X,$ i.e. the dimension
of $H^{0}(X,kK_{X}).$ In the strictly positive case $\kappa(X)$
may be equivalently defined as the the dimension of the image of $X$
under the $k-$th canonical rational mappings 
\[
F_{k}:\,X\dashrightarrow\P^{*}H^{0}(X,kK_{X}),\,\,\,\,Y_{k}:=\overline{F_{k}(X)},\,\,\,\,\,x\mapsto\{s_{k}\in H^{0}(X,kK_{X}):\,s_{k}(x)=0\}
\]
where here and subsequently $k$ stands for a sufficiently large,
or sufficiently divisible, positive integer. By construction $kK_{X}$
is trivial along the fibers of $F_{k}.$ Next, we recall that by classical
results of Iitaka there exist non-singular varieties $X'$ and $Y'$
and a subjective morphism $F$ with connected fibers: 
\[
F:\,X'\rightarrow Y'
\]
such that $X'$ and $Y'$ are birational to $X$ and $Y,$ respectively
and such that $F$ is conjugate to $F_{k}.$ The fibration defined
by $F$ is uniquely determined up to birational equivalence and usually
referred to as the \emph{Iitaka fibration}. A very general fiber of
the fibration has vanishing Kodaira dimension.

Finally, it should be pointed out that by the deep results in \cite{bi,si},
proved in the context of the MMP, the canonical ring $R(X):=\bigoplus_{k\in\N}H^{0}(X,kK_{X})$
of any non-singular projective variety $X$ is finitely generated.
In particular, $Y_{k}$ (as defined above) is, for $k$ sufficiently
divisible, independent of $k$ (up to isomorphism) and coincides with
the\emph{ canonical model} of $X$ (i.e. the Proj of $R(X)).$ But
this information will not be needed for our arguments.

\subsubsection{Canonical point processes on varieties of positive Kodaira dimension}

Let us now consider a non-singular variety $X$ of\emph{ positive
Kodaira dimension }(there is also a logarithmic version of this setup
concerning klt pairs $(X,D)$, but for simplicity we will assume that
$D=0).$\footnote{More generally, the results will apply to $X$ a possibly singular
normal variety, by defining the corresponding probability measures
on the regular locus of $X$ and using the birational invariance below
to replace $X$ with any resolution.} On such a variety $X$ we can define the canonical random point processes
just as in section \ref{sub:Canonical-point-processes klt} (since
$N_{k}>0$ for $k$ large). 
\begin{prop}
The canonical random point processes attached to a variety $X$ of
positive Kodaira dimension are birationally invariant in the sense
that if $F:\,X\dashrightarrow X'$ is a birational mapping, then the
canonical probability measures on $X^{N_{k}}$ and $X'^{N_{k'}}$
are invariant under $F_{*}.$\end{prop}
\begin{proof}
This follows from the usual proof of the birational invariance of
the spaces $H^{0}(X,kK_{X}).$ Indeed, $F$ defines an isomorphism
from $U$ in $X$ to $U'$ in $X',$ where $U$ has codimension at
least two. Hence, by the usual unique extension properties of holomorphic
sections $F_{|V}^{*}$ induces an isomorphism between $H^{0}(X,kK_{X})$
and $H^{0}(X',kK_{X'}),$ which (by the change of variables formula)
respects the measure $\left(S_{k}\wedge\overline{S_{k}}\right)^{1/k}$
defined by an element $S_{k}\in H^{\text{0}}(X,kK_{X}).$ Applying
this argument on the products $X^{N_{k}}$ and $X^{'N_{k}}$ then
concludes the proof.
\end{proof}
When studying the random point processes on $X$ we may without loss
of generality, by the previous proposition, assume that there is morphisms
$F$ of $X$ to the base $Y$ of the Iitaka fibration. There is a
canonical family of relative measures $\mu_{X/Y}$ defined over an
open dense subset $Y_{0}$ of $Y,$ such that $Y-Y_{0}$ is a null
set, defined a follows. First, by the construction of the Iitaka fibration,
we may assume that $F$ is a submersion over some open dense subset
$Y_{0}$ of $Y$ and that $H^{0}(X_{y},kK_{X_{y}})$ is one-dimensional
for $y\in Y_{0}.$ Letting $\Omega_{y}^{(k)}$ denote a generator
of the latter one-dimensional vector space, 
\begin{equation}
(\mu_{X/Y})_{y}:=\left(\Omega_{y}^{(k)}\wedge\overline{\Omega_{y}^{(k)}}\right)^{1/k}/\int_{X_{y}}\left(\Omega_{y}^{(k)}\wedge\overline{\Omega_{y}^{(k)}}\right)^{1/k}\label{eq:def of rel prob meas}
\end{equation}
is a probability measure on $X_{y}$ which is independent of the generator
and of $k$ (since $(\Omega_{y}^{(k)})^{\otimes m}$ generates $H^{0}(X_{y},kmK_{X_{y}}))$
and it defines a smooth family of relative $(n-\kappa,n-\kappa)-$
forms over $Y_{0}.$ Let us also introduce some further notation:
if $\nu_{Y}$ is a measure on the base $Y$ which is absolutely continuous
with respect to Lebesgue measure, then we will write $F^{*}\nu_{Y}\wedge\mu_{X/Y}$
for the measure on $X$ defined as a fiber-product, i.e. if $u$ is
a smooth function on $X$ then
\[
\int_{X}F^{*}\nu_{Y}\wedge\mu_{X/Y}u:=\int_{Y_{0}}\left(\int_{X_{y}}u\mu_{X/Y}\right)\nu_{Y}
\]
(which is independent of the choice of $Y_{0}$ since the complement
is a null set). 
\begin{lem}
\label{lem:f-m}Let $X$ be a variety of positive Kodaira dimension
and assume that the Iitaka fibration \textup{$F:\,X\rightarrow Y$
is a morphism and that the branch locus $Y-Y_{0}$ is equal to the
support of a divisor $D$ in $Y$ with normal crossings. Then there
exists a line bundle $L_{X/Y}$ over $Y$ equipped with a (singular)
metric whose weight will be denoted by $\phi_{H},$ with the property
that $K_{Y}+L_{X/Y}$ is big and for any $S_{k}\in H^{0}(X,kK_{X})$
there exists a unique $s_{k}\in H^{0}(Y,k(K_{Y}+L_{X/Y})$ such that
\begin{equation}
\left(S_{k}\wedge\overline{S_{k}}\right)^{1/k}=F^{*}\left(\left(s_{k}\wedge\overline{s_{k}}\right)^{1/k}e^{-\phi_{H}}\right)\wedge\mu_{X/Y}\label{eq:measures in lemma f-m}
\end{equation}
over $Y_{0}.$ The weight $\phi_{H}$ is smooth on $Y_{0}$ and locally
around any given point in $Y-Y_{0}$ }
\begin{equation}
\phi_{H}=-q\log\log(\left|s_{D}\right|^{-2})+\log\left|s_{D_{X/Y}}\right|^{2}+O(1)\label{eq:sing structure in lemma f-m}
\end{equation}
 for some positive number $q,$ where $D_{X/Y}$ is a klt divisor
whose support coincides with $D.$ The line bundle $L_{X/Y}$ will
be referred to as the Hodge line bundle and $\phi_{H}$ as the weight
of the Hodge metric. \end{lem}
\begin{proof}
By assumption the morphism $F$ restricts to define a submersion $\pi:X_{0}\rightarrow Y_{0}$
between Zariski open subsets. The (tautological) decomposition $K_{X}=F^{*}K_{Y}+K_{X/Y}$
restricted over $Y^{0}$ gives 
\begin{equation}
kK_{X_{0}}=F^{*}k\left(K_{Y_{0}}+F^{*}L_{0}\right),\label{eq:proof of lemma f-m}
\end{equation}
 where $L_{0}:=\pi_{*}(K_{X_{0}/Y_{0}}).$ The latter direct image
sheaf is defined as a $\Q-$line bundle over $Y_{0}:$ $F_{*}(K_{X_{0}/Y_{0}})=\frac{1}{k}F_{*}(kK_{X_{0}/Y_{0}})$
for any fixed $k$ which is sufficiently large. Concretely, $kL_{0}$
is locally trivialized by $\Omega_{y}^{(k)},$ where $\Omega_{y}^{(k)}$
is as in formula \ref{eq:def of rel prob meas}. We equip the direct
image line bundle $L_{0}$ with the canonical $L^{2}-$metric, usually
referred to as the Hodge metric. Concretely, the $k$ tensor power
of the letter metric is defined by $\left\Vert \left(\Omega_{y}^{(k)}\right)\right\Vert ^{2}:=\left(\int_{X_{y}}\left(\Omega_{y}^{(k)}\wedge\overline{\Omega_{y}^{(k)}}\right)^{1/k}\right)^{k},$
i.e the local weight $\phi_{H}$ for $L_{0}$ determined by the trivialization
(multi-) section $\left(\Omega_{y}^{(k)}\right)^{1/k}$ is given by
\[
\phi_{H}(y):=-\log\left\Vert \left(\Omega_{y}^{(k)}\right)^{1/k}\right\Vert ^{2}=-\log\int_{X_{y}}\left(\Omega_{y}^{(k)}\wedge\overline{\Omega_{y}^{(k)}}\right)^{1/k})
\]
Next, fixing $x\in X_{0}$ we take a small neighborhood $V$ of $y_{0}:=F(x_{0})$
and local holomorphic coordinates $w$ centered at $y_{0}$ and set
$dw:=dw_{1}\wedge\cdots\wedge dw_{\kappa}.$ Then the restriction
of $S_{k}\in H^{0}(X,kK_{X})$ to $U:=F^{-1}(V)$ may be written as
$S_{k}=f_{k}(w)\Omega_{y}^{(k)}\otimes dw^{\otimes k}$ for a local
holomorphic function $f_{k}(w)$ on $V$ (which transforms as a section
of $kL_{0}\rightarrow Y_{0})$ and 

\[
\left(S_{k}\wedge\overline{S_{k}}\right)^{1/k}=\left|f_{k}(w)\right|^{2/k}dw\wedge d\bar{w}\wedge\Omega_{y}^{(k)}\wedge\overline{\Omega_{y}^{(k)}}=\left|f_{k}(w)\right|^{2/k}e^{-\phi_{H}(w)}dw\wedge d\bar{w}\wedge\mu_{X/Y},
\]
 where $\mu_{X/Y}$ is the relative probability measure defined in
formula \ref{eq:def of rel prob meas}. Since $x$ was an arbitrary
point in $X_{0}$ this proves the relation \ref{eq:measures in lemma f-m}
over $Y_{0}$ if $s_{k}$ is taken in $H^{0}(Y_{0},k(K_{Y}+L_{0}).$ 

Next, we will give the construction of the line bundle $L_{X/Y}$
extending $L_{0}$ and show that $s_{k}$ above can be taken as the
restriction to $Y_{0}$ of an element in $H^{0}(Y,k(K_{Y}+L_{X/Y}).$
First, following Fujino-Mori \cite{f-m}, we may assume that the double
dual of the torsion free sheaf $\pi_{*}(kK_{X/Y})/k$ is a well-defined
$\Q-$line bundle and set $L_{X/Y}:=L_{0}.$ The canonical bundle
formula of Fujino-Mori (see Prop 2.2. in \cite{f-m}) says that 
\[
K_{X}+B_{-}=\pi^{*}(K_{Y}+L_{X/Y})+B_{\text{+}}
\]
 where $B_{\pm}$ are effective $\Q-$divisors (supported in $Y-Y_{0})$
such that $\mbox{codim}F(\mbox{supp}(B_{-}))\geq2$ and $F_{*}(\mathcal{O}(kB_{+})=\mathcal{O}_{Y}.$
This formula implies that
\begin{itemize}
\item If $S_{k}\in H^{0}(X,kK_{X}),$ then the restriction of $S_{k}$ to
$X_{0}$ may be written as $S_{k|X_{0}}=F^{*}s_{k|Y_{0}}$ for a unique
section $s_{k}\in H^{0}(Y,k(K_{Y}+L_{X/Y})).$ 
\end{itemize}
Indeed, by the canonical bundle formula and the property of $B_{+},$
the restriction of $S_{k}$ to $Y-\mbox{supp}(B_{-})$ may be written
as $S_{k}=F^{*}s_{k}\otimes s_{B_{+}}^{\otimes k}$ for a unique $s_{k}\in H^{0}(Y-F(\mbox{supp}(B_{-})),kK_{Y}).$
But since $\mbox{codim}F(\mbox{supp}(B_{-}))\geq2$ the section $s_{k}$
extends to a unique element in $H^{0}(Y,k(K_{Y}+L_{X/Y})).$ Since
$\mbox{supp}(B_{-})\subset Y-Y_{0}$ this proves the point above.
Note that this is essentially the same argument as the one used in
\cite{f-m} to prove that $H^{0}(X,kK_{X})=H^{0}(X,k(K_{Y}+L_{X/Y})$
(see the proof of Theorem 4.5 in op.cit.) and it also shows that $K_{Y}+L_{X/Y}$
is big (since $N_{k}\sim k^{\kappa}).$ Finally, let us briefly recall
the proof of the singularity structure of the Hodge metric on $L_{X/Y}$
in a neighborhood of a point contained in $Y-Y_{0},$ which follows
from Tsuji's argument in \cite{ts}. First, as shown in \cite{f-m}
may assume that $L_{X/Y}=M_{X/Y}+D_{X/Y},$ where $M_{X/Y}$ (``
the semi-stable part'') is a nef line bundle on $Y$ and $D_{X/Y}$
(the ``discriminant part'') is a klt divisor on $Y.$ By the work
of Kawamata and Schmidt on variations of Hodge structures (see \cite{ts}
and references therein) the lines bundles $M_{X/Y}$ and $D_{X/Y}$
contribute over $Y-Y_{0}$ to the first and second term in the decomposition
\ref{eq:sing structure in lemma f-m} of the weight $\phi_{H}$ defined
with respect to a given trivialization of $L_{X/Y}$ over a neighborhood
of a point in $Y-Y_{0}.$
\end{proof}
Next, let us recall the definition of the (singular) canonical metric
$\omega_{Y}$ on the base of the base $Y$ of the Iitaka fibration
(which is a birational invariant). For our purposes it will be enough
to define it in the case when $X$ fibers over $Y$ as in the previous
lemma. Then we define $\omega_{Y}\in c_{1}(K_{Y}+L_{X/Y})$ as $\omega_{Y}=dd^{c}\phi_{Y},$
where $\phi_{Y}$ is the weight of a (possible singular) positively
curved metric $\phi_{Y}$ on $K_{X}+L_{X/Y}$ defined as the unique
finite energy weight $\phi$ on the big line bundle $K_{X}+L_{X/Y}$
solving 
\begin{equation}
MA(\phi)=e^{\phi-\phi_{H}}idw\wedge d\bar{w.}\label{eq:ma eq for phi on base of ithaka}
\end{equation}
(which is an equation of the form appearing in Prop \ref{prop:min of free energy}).Note
that by the previous lemma we have that locally $e^{-\phi}\in L_{loc}^{p}$
for some $p>1$ and hence by the Kolodziej type estimates in \cite{begz}
$\phi$ has minimal singularities. We may hence define $\omega_{Y}$
equivalently as the unique current in $c_{1}(K_{Y}+L_{X/Y})$ with
minimal singularities such that 

\[
\mbox{Ric \ensuremath{\omega_{Y}=-\omega_{Y}+\omega_{WP}+[\Delta]}},
\]
 where $\Delta$ is the klt divisor of Fujino-Mori (the ``discriminant
divisor'') supported on the branch locus in $Y$ and $\omega_{WP}$
is equal to $1_{Y_{0}}\omega_{WP}$ where $\omega_{WP}$ is the generalized
Weil-Petersson type metric of the fibration over $Y_{0}$ (compare
\cite{s-t}). It may be defined as the curvature form of the Hodge
metric on $L_{X/Y}\rightarrow Y_{0}.$ Alternatively, we note that
arguing as in the beginning of the proof of the previous lemma  gives
the following equivalent equation for $\phi_{Y},$ where the pull-back
and push-forward is defined over $Y_{0}$ (and then extended by zero):
\[
MA(\phi_{Y})=F_{*}(e^{F^{*}\phi_{Y}}dz\wedge d\bar{z})
\]
(where the right hand side is defined with respect to suitable local
coordinates $(z,w)$ on $X$ respecting the fibration). Finally, we
define the canonical probability measure$\mu_{X}$ on $X$ of Song-Tian-Tsuji
as 
\begin{equation}
\mu_{X}:=F^{*}\mu_{Y}\wedge\mu_{X/Y}\label{eq:def of mu X text}
\end{equation}
Equivalently, $\mu_{X}=F^{*}\omega_{Y}^{\kappa}\wedge\omega_{CY}^{n.\kappa},$
where $\omega_{CY}$ denotes a family of Ricci flat Kähler metrics
defined over the very general Calabi-Yau fibers and the metrics are
normalized to have unit-volume (abusing terminology slightly the term
Calabi-Yau refers to a variety of zero Kodaira dimension). 
\begin{thm}
\label{thm:pos kod text}Let $X$ be projective variety of positive
Kodaira dimension. Then the empirical measures of the canonical random
point processes on $X$ converge in probability towards the canonical
probability measure $\mu_{X}$ of Song-Tian-Tsuji.\end{thm}
\begin{proof}
By the previous lemma we can write$\mu^{(N_{k})}=F^{*}\mu_{Y}^{(N_{k})}\wedge\mu_{X/Y}^{\otimes N_{k}}$
where $\mu_{Y}^{(N_{k})}$ is defined with respect to the big line
bundle $K_{Y}+L_{X/Y}\rightarrow Y$ equipped, where $L_{X/Y}$ is
equipped with the Hodge metric. In particular, the $j-$point correlation
measures $(\mu^{(N_{k})})_{j}$ are given by 
\[
(\mu^{(N)})_{j}:=\int_{X^{N-j}}\mu^{(N)}=F^{*}(\mu_{Y}^{(N_{k})})_{j}\wedge\mu_{X/Y}^{\otimes j}
\]
Next we note that we may proceed as in the proof of Theorem \ref{thm:big line bundle text}
to see that the empirical measures of $\mu_{Y}^{(N_{k})}$ converge
in probability towards $\mu_{Y}$ (and even with with a LDP), which
implies that
\[
(\mu_{Y}^{(N_{k})})_{j}\rightarrow\mu_{Y}^{\otimes j}.
\]
(compare Cor \ref{cor:conv of j point in general setting}). Indeed,
fixing smooth Hermitian metrics (weights) $\phi$ and $\phi_{0}$
on $K_{Y}$ and $L_{X/Y}$ the measure $\left(s_{k}\wedge\overline{s_{k}}\right)^{1/k}e^{-\phi_{H}}$
defined by an element $s_{k}\in H^{0}(Y,k(K_{Y}+L_{X/Y})$ may be
written as $\left\Vert s_{k}\right\Vert ^{2}\mu_{\phi}e^{-(\phi_{H}-\phi_{0})}$
and by the previous lemma 
\[
\mu_{\phi}e^{-(\phi_{H}-\phi_{0})}=e^{-v}\mu_{\Delta}:=\mu_{0}
\]
for a klt divisor $\Delta$ where $v$ is upper semi-continuous and
such that $e^{-v}\mu_{\Delta}$ is a finite measure (even with an
$L^{p}-$density for some $p>1).$ But then we may proceed precisely
as in the proof of Theorem \ref{thm:big line bundle text}: the argument
for the upper bound works the same for any finite measure $\mu_{0}$
and for the proof of the lower bound we can take a sequence $v_{j}$
of continuous functions decreasing to $v$ and replace $\mu_{0}$
with $\mu_{j}:=e^{-v_{j}}\mu_{\Delta}.$ Then we let $j\rightarrow\infty$
in the end of the argument, just as before and obtain the desired
convergence in probability towards the deterministic measure $\mu_{Y}$
satisfying $\mu_{Y}=MA(\phi_{Y}),$ where $MA(\phi_{Y})=e^{\phi_{Y}-\phi_{H}}idw\wedge d\bar{w,}$
as desired.
\end{proof}
Finally, combining the previous convergence with Sanov's theorem  gives
\[
(\mu^{(N)})_{j}\rightarrow F^{*}(\mu_{Y}^{\otimes j})\wedge\mu_{X/Y}^{\otimes j}=(F^{*}(\mu_{Y})\wedge\mu_{X/Y})^{\otimes j},
\]
 which equivalently means that the canonical empirical measures on
$X$ converge in probability towards $(F^{*}(\mu_{Y})\wedge\mu_{X/Y},$
as desired. 
\begin{cor}
\label{cor:conv of caonical currents towards pull-back}Let $X$ be
projective variety of positive Kodaira dimension. Then the currents
\[
\omega_{k}:=\frac{i}{2\pi}\partial\bar{\partial}\log\int_{X^{N_{k}-1}}\left|(\det S^{(k)})(\cdot,z_{1},...,z_{N_{k}-1})\right|^{2/k}dz_{1}\wedge d\bar{z}_{1}\wedge\cdots\wedge dz_{N_{k}-1}\wedge d\bar{z}_{N_{k}-1}
\]
converge, as $k\rightarrow\infty,$ weakly towards a canonical positive
current in $c_{1}(K_{X})$ which, on a Zariski open subset coincides
with $F^{*}\omega_{Y},$ i.e. the pull-back of the canonical metric
on the base of the Iitaka fibration.\end{cor}
\begin{proof}
Arguing exactly as in the proof of Cor \ref{cor:conv of canoc seq of currents klt}
gives that $\omega_{k}\rightarrow dd^{c}\log(\mu_{X}).$ But specializing
formula \ref{eq:def of mu X text} to $X_{0}$ over $Y_{0}$ gives
that $dd^{c}\log(\mu_{X})=dd^{c}F^{*}\omega_{Y}+0$ on $X_{0},$ using
that $dd^{c}\log\left|\Omega_{y}^{(k)}\right|^{2}=0\text{ }$ and
that the terms involving $dd^{c}\phi_{H}$ cancel (as is seen by working
on a local set $U$ in $X_{0}$ as in the beginning of the proof of
the previous lemma).\end{proof}
\begin{rem}
If $K_{X}$ is semi-ample we can take $F$ to be the morphism defined
by the canonical map at some fixed level $k$ so that $Y:=X_{can}$
is the canonical model of $X.$ Then we can define $\omega_{can}$
as $\omega_{can}:=dd^{c}\phi_{can}$ where $\phi_{can}$ is the unique
(locally bounded) positively curved metric on $\mathcal{O}(1)_{|X_{can}}$
solving the equation \ref{eq:ma eq for phi on base of ithaka} (using
that $\mathcal{O}(1)$ is naturally isomorphic to $K_{Y}+L_{X/Y}$
over $Y_{0}).$ The metric $\omega_{can}$ thus defined yields a canonical
Kähler metric in $c_{1}(\mathcal{O}(1)_{|X_{can}})$ which, by the
uniqueness argument in \cite{s-t} coincides with the canonical metric
constructed in \cite{s-t}. Accordingly, the limiting current obtained
in the previous corollary coincides with $F^{*}\omega_{can}$ on $X_{0}$
and hence everywhere since the currents are elements in the same cohomology
class $c_{1}(K_{X}).$ In this setting the previous corollary is thus
analogous to the convergence result for the Kähler-Ricci flow for
a variety with $K_{X}$ semi-ample obtained in \cite{s-t}.
\end{rem}

\section{\label{sec:Fano-manifolds-and}Fano manifolds and Gibbs stability }

In this section we will outline a conjectural general picture concerning
the the case when $\beta=-\text{1 }$in the Kähler-Einstein equation,
i.e. the case of Kähler-Einstein metrics with positive Ricci curvature.
In other words, this the case when the dual $-K_{X}$ of the canonical
line bundle is ample, which means that $X$ is a \emph{Fano manifold.}
We will establish a weak form of the conjecture, but we leave the
general case for the future. 

If a Kähler-Einstein metric exists on a Fano manifold $X$ then, by
the Bando-Mabuchi theorem, it is uniquely determined modulo the action
of automorphism group generated by holomorphic vector fields on $X.$
But in general there are obstructions to the existence of a Kähler-Einstein
metric on $X$ and according to the Yau-Tian-Donaldson conjecture
$X$ admits a Kähler-Einstein metric precisely when $X$ is K-polystable.
This latter notion of stability is of an algebro-geometric nature
and can be formulated as an asymptotic version of stability in the
sense of Geometric Invariant Theory (GIT). Recently, the conjecture
was finally settled by Chen-Donaldson-Sun \cite{c-d-s} (see also
\cite{ti}). Here we will introduce a probabilistic/statistical mechanical
version of the Yau-Tian-Donaldson where the notion of K-stability
is replaced by a notion that we will call\emph{ Gibbs stability. }To
explain this first observe that to be able to define an analog of
the probability measure appearing in formula \ref{eq:canon prob measure intro}
in the introduction of the paper to the Fano setting we have to replace
$K_{X}$ with its dual $-K_{X}$ to ensure the existence of holomorphic
sections. But this forces us to replace the exponent $1/k$ with $-1/k,$
in order to get a well-defined density on $X^{N_{k}}.$ However, there
is then no guarantee that the corresponding normalization constant
\begin{equation}
Z_{N_{k}}:=\int_{X^{N_{k}}}\left|(\det S^{(k)})(z_{1},...,z_{N_{k}})\right|^{-2/k}dz_{1}\wedge d\bar{z}_{1}\wedge\cdots\wedge dz_{N_{k}}\wedge d\bar{z}_{N_{k}}\label{eq:canonincal partition f in fano case}
\end{equation}
 is finite, since the corresponding integrand is singular along the
zero-locus of $\det S^{(k)}.$ Accordingly, we say that a Fano manifold
$X$ is \emph{Gibbs stable at level $k$} if $Z_{N_{k}}$ is finite
(which is independent of the choice of generator $\det S^{(k)}$)
and \emph{asymptotically Gibbs stable} if it is Gibbs stable at level
$k$ for any $k,$ sufficiently large. 
\begin{conjecture}
\label{conj: fano}Let $X$ be Fano manifold. Then $X$ admits a unique
Kähler-Einstein metric $\omega_{KE}$ if and only if $X$ is asymptotically
Gibbs stable. Moreover, if $X$ is asymptotically Gibbs stable, then
the empirical measures of the corresponding point processes converge
in probability towards the normalized volume form of $\omega_{KE}$
and the corresponding canonical sequence of curvature forms $\omega_{k}$
converge weakly towards $\omega_{KE}.$
\end{conjecture}
Here $\omega_{k}$ is defined by 
\[
-\frac{i}{2\pi}\partial\bar{\partial}\log\int_{X^{N_{k}-1}}\left|(\det S^{(k)})(\cdot,z_{1},...,z_{N_{k}-1})\right|^{-2/k}dz_{1}\wedge d\bar{z}_{1}\wedge\cdots\wedge dz_{N_{k}-1}\wedge d\bar{z}_{N_{k}-1}.
\]
It follows from applying Berndtsson's positivity of direct images
\cite{bern1} $N-1$ times (i.e. for each factor of $X^{N-1})$ that
$\omega_{k}$ is in fact positive in the sense of currents and hence
a Kähler form for $k$ sufficiently large. One can also define a weaker
notion of Gibbs stability which takes automorphisms into account,
but for simplicity we will focus here on the case when $X$ admits
no automorphisms.

Interestingly, the notion of Gibbs stability introduced above can
also be given the following purely\emph{ algebro-geometric }formulation
in the spirit of the Minimal Model Program: let $\mathcal{D}_{k}$
be the effective divisor in $X^{N_{k}}$ cut out by the section $\mbox{det \ensuremath{S^{(k)}.}}$
Geometrically, $\mathcal{D}_{k}$ may be represented as the following
incidence divisor in $X^{N_{k}}:$ 
\[
\mathcal{D}_{k}:=\{(x_{1},...x_{N})\in X^{N_{k}}:\,\exists s\in H^{0}(X,-kK_{X}):\,s(x_{i})=0,\,i=1,..,N_{k}\}
\]
 Gibbs stability at level $k$ amounts to saying that $\mathcal{D}_{k}/k$
is mildly singular in the sense of MMP (i.e. its singularities are
Kawamata Log Terminal) or more precisely that 
\begin{equation}
\mbox{lct\ensuremath{(\mathcal{D}_{k}/k)>1}}\label{eq:bound on lct}
\end{equation}
for $k>>1,$ where $\mbox{lct\ensuremath{(\mathcal{D}_{k}/k)}}$denotes
the the\emph{ log canonical threshold (lct)} of the anti-canonical
$\Q-$divisor divisor $\mathcal{D}_{k}/k$ on $X^{N_{k}}.$ The equivalence
with the original definition follows directly from the\emph{ analytic
}definition of the lct of a divisor $D=\{s=0\}$ as the sup of all
$t$ such that $1/|s|^{2t}$ is locally integrable (also using the
``openness property'' for the lct of algebraic singularities). It
also seems natural to say that $X$ is \emph{uniformly Gibbs stable,}
if 
\begin{equation}
\gamma(X):=\liminf_{k\rightarrow\infty}\mbox{lct}(\mathcal{D}_{k}/k)>1.\label{eq:invariant gamma}
\end{equation}
There is also an even stronger notion of Gibbs stability which we
call\emph{ strong Gibbs stability }and which demands that 
\[
\lim_{k\rightarrow\infty}\frac{1}{N_{k}}\log Z_{N_{k},\beta}<\infty
\]
for some $\beta<1,$ where $Z_{N_{k},\beta}$ is the partition function
at inverse temperature $\beta$ (formula \ref{eq:def of partion function in beta-setting text})
defined with respect to a metric $\left\Vert \cdot\right\Vert $ on
$-K_{X},$ a volume form $dV$ (for example the one defined by the
metric $\left\Vert \cdot\right\Vert $) and the corresponding generator
$\det S^{(k)}.$ 
\begin{prop}
Suppose that $X$ admits non-trivial holomorphic vector fields. Then
$\gamma(X)\leq1.$ More precisely, for $k$ sufficiently large the
canonical partition function $Z_{N_{k}}$ is non-finite.\end{prop}
\begin{proof}
Assume to get a contradiction that $Z_{N_{k}}$ is finite. Then we
may define a probability measure $\nu_{k}$ on $X$ as the corresponding
one-point correlation measure. But, by construction, $\nu_{k}$ is
invariant under the action of the automorphism group of $X.$ Moreover,
for $k$ sufficiently large $H^{0}(X,-kK_{X})$ is basepoint free
and hence $\nu_{k}$ then defines a volume form on $X.$ But it follows
from general principles that if a complex manifold $X$ admits a non-trivial
holomorphic vector field $v$ than it cannot admit an automorphism
invariant volume form. This is proved in a standard way by taking
the Lie derivative of $\nu_{k}$ along the vector field $v$ and using
Cartan's formula. 
\end{proof}
Next, we will show the following weak partial confirmation of the
conjecture above.
\begin{thm}
Suppose that the Fano manifold $X$ is strongly Gibbs stable. Then
it admits a unique Kähler-Einstein metric. More generally, the corresponding
result holds in the setting of (possibly singular) log Fano varieties.\end{thm}
\begin{proof}
Fix a volume form $\mu$ on $X.$ By the Gibbs variational principle
(i.e. Jensen's inequality) 
\[
-\frac{1}{N_{k}\beta}\log Z_{N_{k},\beta}\leq E^{(N)}(\mu^{\otimes N})+\frac{1}{\beta}D_{\mu_{0}}(\mu),
\]
 where, for any given symmetric probability measure $\mu_{N}$ on
$X^{N}$ 
\[
E^{(N)}(\mu_{N}):=\int_{X^{N}}\frac{H^{(N)}}{N}\mu_{N},
\]
i.e. the mean $N-$particle energy. Setting $\beta=-(1+\delta)$ for
some $\delta>0$ and using the definition of Gibbs stability thus
gives 
\[
-(1+\delta)E^{(N)}(\mu^{\otimes N})+D_{\mu_{0}}(\mu)\geq-C
\]
 We will conclude the proof by observing that
\begin{equation}
\liminf_{N\rightarrow\infty}E^{(N)}(\mu^{\otimes N})\geq E_{\theta}(\mu).\label{eq:lower bound on energy}
\end{equation}
 Accepting this for the moment gives that $(1+\delta)E(\mu)+D_{\mu_{0}}(\mu)\geq-C,$
which, by definition, means that the canonical free energy functional
$F$ (for $\beta=1)$ is coercive. But since the latter functional
may be identified with Mabuchi's K-energy functional it then follows
from a result of Tian that $X$ admits a Kähler-Einstein metric (which
by the coercivity has to be unique, since the coercivity rules out
automorphisms). More generally, the existence of a minimizer of the
functional $F,$ satisfying the corresponding Monge-Ampère equations
was shown in \cite{berm6,bbgez} in various singular settings, in
particular in the setting of log Fano varieties \cite{bbgez}. Finally,
let us prove the lower bound \ref{eq:lower bound on energy}. The
proof is similar to the upper bound in the proof of Theorem \ref{thm:big line bundle text}.
By definition, for any given $u\in C^{0}(X)$
\[
E^{(N)}(\mu^{\otimes N})=\int_{X^{N}}\frac{H^{(N)}}{N}\mu^{\otimes N}=\int_{X^{N}}\frac{H^{(N)}+u}{N}\mu^{\otimes N}-\int_{X}u\mu
\]
Hence, estimating $\frac{H^{(N)}+u}{N}$ from below, using the first
point in Theorem \ref{thm:thm A and B in b-b} and taking the sup
over all $u\in C^{0}(X)$ proves \ref{eq:lower bound on energy} (also
using Prop \ref{prop:pluri energy as legendre} in the last step).
\end{proof}
It may very well be that in the end all the notions of Gibbs stability
introduced above will turn out to be equivalent. For the moment the
author has only been able to prove this in the the first non-trivial
setting of one dimensional log Fano manifolds, where the analog of
the Conjecture \ref{conj: fano} indeed holds (the proof will appear
in a separate publication).

It seems also natural to conjecture that the following invariant $\gamma(X)$
defined by \ref{eq:invariant gamma} satisfies 
\[
\min\left\{ \gamma(X),1\right\} =R(X),
\]
 where $R(X)$ is the greatest lower bound on the Ricci curvature
\cite{sz}, i.e. the sup over all $r\in[0,1]$ such that there exists
a Kähler form $\omega\in c_{1}(-K_{X})$ satisfying $\mbox{Ric \ensuremath{\omega\geq r\omega.}}$
As support for the latter conjecture we note that the inequality $R(X)\geq\min\left\{ \gamma(X),1\right\} $
follows from a simple modification of the proof of the previous theorem.
Finally, let us point out that these conjectures are related to another
conjectural property, namely that the reversed inequality in formula
\ref{eq:lower bound on energy} holds, i.e. that
\begin{equation}
\lim_{N\rightarrow\infty}E^{(N)}(\mu^{\otimes N})=E_{\theta}(\mu)\label{eq:conj for energy}
\end{equation}
 for any volume form on $\mu$ (see the discussion in \cite[Section 6]{berm8}).
In one dimension this follows from the bosonization formula and it
also holds in the toric setting when $\mu$ is a torus invariant measure
(as follows form the results in \cite{berm7}). The validity of formula
\ref{eq:conj for energy} would imply the following approximation
result for the Calabi-Yau equation \cite{y}: given a normalized volume
form $dV$ on $X$ and an ample line bundle $L\rightarrow X$ the
Kähler metrics 
\[
\omega_{k}:=\frac{i}{\pi}\partial\bar{\partial}\frac{1}{k}\int_{X^{N_{k}-1}}\log\left|\det S_{k}(\cdot,x_{1},...,x_{N_{k}-1})\right|dV^{\otimes N_{k}-1}\in c_{1}(L)
\]
 converge weakly towards the unique Kähler form $\omega$ in $c_{1}(L)$
with volume form $dV.$ This conjectural approximation result is the
Kähler analogue of the approximation for optimal transport maps obtained
in \cite{berm7} in a convex analytical setting.
\begin{rem}
Since the appearance of the first preprint version of the present
paper, where the conjectures above were formulated, there has been
some progress that we briefly recall. First, it was shown by Fujita
that uniform Gibbs stability implies K-stability and hence by \cite{c-d-s}
the existence of a Kähler-Einstein metric on $X,$ which thus partly
confirms Conjecture \ref{conj: fano} above (when using the notion
of uniform Gibbs stability). The proof in \cite{fuj} proceeds by
an elegant algebraic argument using Odaka's formulation of K-stability
in terms of equivariant multiplier ideal sheaves together with Mustata's
summation formula for multiplier ideal sheaves. More generally, very
recently it was shown in \cite{f-o} that uniform Gibbs stability
implies uniform K-stability (and hence the existence of a Kähler-Einstein
metric could alternatively be deduced from the recent variational
approach to the Yau-Tian-Donaldson conjecture introduced in \cite{bbj},
showing that the existence of a Kähler-Einstein metric on a Fano manifold,
not admitting non-trivial holomorphic vector fields, is equivalent
to uniform K-stability). Finally, let us also mention that a real
(tropical) analog of Conjecture \ref{conj: fano} in the setting of
toric Fano manifolds, where the role of determinants is played by
permanents, has been established in \cite{b-o}. It implies, in particular,
the inequality $\min\left\{ \gamma(X),1\right\} \leq R(X)$ when $X$
is toric.\end{rem}

\end{document}